\documentclass[article]{IEEEtran}
\usepackage[english]{babel}
\usepackage[latin1]{inputenc}
\usepackage{enumerate}
\usepackage{color}
\usepackage[T1]{fontenc}
\usepackage{epstopdf}
\usepackage{subfigure}
\usepackage{dsfont}
\usepackage[T1]{fontenc}
\usepackage{amsmath}
\usepackage{mathtools}
\usepackage{amsthm}
\usepackage{amstext}
\usepackage{amssymb}
\usepackage{mathrsfs}
\usepackage{cite}
\usepackage{mathtools}
\usepackage{tikz}
\usetikzlibrary{arrows,positioning}

\def\R{\mathbb{R}}

\def\diam{\mathop{\rm diam}}
\def\Lip{\mathop{\rm Lip}}

\def\T{ T}
\def\L{ L}

\def\intr{\mathop{\rm int}}

\def\argmin{\mathop{\rm arg\, min}}

\def\B{{\mathcal B}}

\def\C{{\mathcal C}}
\def\V{{\mathcal V}}
\def\P{{\mathcal P}}

\def\K{{\mathcal K}}

\def\M{{\mathcal M}}
\def\S{{\mathcal S}}
\def\W{{\mathcal W}}

\def\balpha{{\boldsymbol \alpha}}
\def\bdelta{{\boldsymbol \delta}}
\def\bk{{\boldsymbol k}}
\def\bm{{\boldsymbol m}}
\def\bh{{\boldsymbol h}}
\def\bd{{\boldsymbol d}}
\def\br{{\boldsymbol r}}

\def\sint{{\mathsf{int}}}

\def\sX{{\mathsf X}}

\def\sA{{\mathsf A}}
\def\sH{{\mathsf H}}

\def\sE{{\mathsf E}}

\def\cE{\mathbb{E}}
\def\sV{{\mathsf V}}

\theoremstyle{remark}
\newtheorem{example}{Example}

\newtheorem{theorem}{Theorem}

\newtheorem{proposition}{Proposition}
\newtheorem{lemma}{Lemma}
\theoremstyle{remark}
\newtheorem{remark}{Remark}
\newtheorem{assumption}{Assumption}

\allowdisplaybreaks

\newcommand{\appsec}{
\renewcommand{\thesubsection}{\Alph{subsection}}
}

\IEEEoverridecommandlockouts

\begin{document}
\sloppy
\title{Finite-State Approximations to Discounted and Average Cost Constrained Markov Decision Processes}
\author{Naci Saldi
\thanks{The author is with the Department of Natural and Mathematical Sciences, Ozyegin University, Cekmekoy, Istanbul, Turkey, Email: naci.saldi@ozyegin.edu.tr}
\thanks{The material in this paper was presented in part at the 53rd Annual Allerton Conference on Communication, Control and Computing, Monticello, Illinois, Oct.\ 2015.}
}
\maketitle

\begin{abstract}
In this paper, we consider the finite-state approximation of a discrete-time constrained Markov decision process (MDP) under the discounted and average cost criteria. Using the linear programming formulation of the constrained discounted cost problem, we prove the asymptotic convergence of the optimal value of the finite-state model to the optimal value of the original model. With further continuity condition on the transition probability, we also establish a method to compute approximately optimal policies. For the average cost, instead of using the finite-state linear programming approximation method, we use the original problem definition to establish the finite-state asymptotic approximation of the constrained problem and compute approximately optimal policies. Under Lipschitz type regularity conditions on the components of the MDP, we also obtain explicit rate of convergence bounds quantifying how the approximation improves as the size of the approximating finite state space increases.
\end{abstract}

\section{Introduction}\label{sec1}

In Markov decision theory, computing an optimal cost (or optimal value) and an optimal policy is in general intractable for systems with uncountable state spaces. Therefore, it is a practically important problem to find computational tools that yield approximately optimal solutions. In the literature, several methods have been developed for unconstrained MDPs to approach this problem: approximate dynamic programming, approximate value or policy iteration, simulation-based techniques, neuro-dynamic programming (or reinforcement learning), state aggregation, etc. The reader is referred to \cite{Roy06,BeTs96,Whi80,Whi79,Ber75,Pow07,ChFuHuMa07,Lan81,Whi78,DuPr12,DuPr13,DuPr14,EsSuKuLy17,SaLiYu13-2,SaYuLi16,SaYuLi17} and references therein for a rather complete summary of these methods.

In most prior works, a canonical way to approximately compute the optimal value and the optimal policy for MDPs with large number of states has been to approximate the dynamic programming equation associated with the control problem by constructing a reduced finite model. This finite model can be obtained through quantizing \cite{Whi78,Whi79} or randomly sampling \cite{DuPr14-a,MuSz08} the state space. Using strong (and often restrictive) continuity and regularity of the operator in the dynamic programming equation, it has been shown that the reduced finite model converges to the original model as the number of quantization bins or sampling points gets larger. Namely, the optimal value of the reduced model converges to the optimal value of the original MDP.

Although approximate dynamic programming is an efficient method, it cannot be applied to constrained Markov decision problems as the dynamic programming principle does not in general hold in the presence of constraints. Therefore, many approximation methods for unconstrained MDPs cannot be applied directly to constrained MDPs. In this paper, our goal is to develop the finite-state approximation procedure for computing approximately optimal values and near optimal policies for \emph{constrained} MDPs under discounted and average cost criteria.

\subsection{Relevant Literature and Contributions}

In the literature, various methods have been developed for the approximation of constrained MDPs. In \cite{Bha10}, the author develops an actor-critic algorithm with function approximation for finite-state constrained MDPs with discounted cost. An online version of this algorithm for the average cost is introduced in \cite{BhLa12}. In \cite{LaBh12}, a Q-learning algorithm with function approximation for finite-state constrained MDPs with the average cost criterion is established. In \cite{HeLa98}, the authors consider approximation of infinite linear programs where their findings can be applied to study approximation of constrained MDPs. In \cite{HaJa15}, the authors consider risk-aware MDPs via the convex analytic approach, where an
aggregation-relaxation-inner approximation method along with discretization method is used to arrive at approximately optimal solutions.

Two notable exceptions in the literature, that consider approximation of uncountable state constrained MDPs for discounted cost criterion, are \cite{DuPr13,DuPr14-a}, where the authors consider finite linear programming approximation of constrained MDP. They establish a sequence of approximate models using, respectively, quantized and empirical distributions of a probability measure $\mu$ with respect to which the transition probabilities of the MDP are absolutely continuous. In these papers, the authors assume that the transition probability is Lipschitz continuous in state-action pair with respect to the Wasserstein distance of order $1$, and the one-stage cost function and the constraint functions are also Lipschitz continuous in state-action pair. Under these conditions, they establish the convergence of the optimal value of the approximate model to the optimal value of the constrained MDP. They also provide a rate of convergence result that quantifies how the approximation improves as the size of the approximating finite model increases. In \cite{EsSuKuLy17}, the authors consider finite linear programming approximations of constrained MDPs for both the discounted cost and average cost criteria. They assume that the set of feasible state-action pairs is the unit hypercube, the transition probability is Lipschitz continuous in state-action pair with respect to the total variation distance, and the one-stage cost function and the constraint functions are also Lipschitz in state-action pair. Under these conditions, they prove the convergence of the optimal values of the finite models to the optimal value of the original model. They also provide explicit error bounds for the approximation of optimal value.

A common limitation of the aforementioned contributions is that they do not establish a method to compute approximately optimal policies using finite models. To be more precise, they do not show that the optimal policies of the finite models are feasible for the original one and the true cost functions of these policies converge to the cost function of the optimal policy of the original model. Therefore, using these results, approximately optimal policies for the original model cannot be computed; instead, one can only obtain approximately optimal values.

Furthermore, in these papers, restrictive Lipschitz type regularity conditions are imposed on the transition probability, the one-stage cost, and the constraint functions. This is often required for being able to obtain explicit rates of convergence for the approximation methods, however in practice one often has much weaker regularity properties and it would be desirable to establish convergence results under weaker conditions (even if rates of convergence may not be attained and only asymptotic optimality may be guaranteed).

In this paper, we establish complementing results to what is present in the literature reviewed above (in particular \cite{DuPr13,DuPr14-a,EsSuKuLy17}). In the first part of the paper, we study the finite-state approximation problem for computing near optimal values and near optimal policies for constrained MDPs with \emph{compact} state and action spaces, under the discounted cost and average cost criteria, where the finite-state models are obtained through the discretization, on a finite grid, of the state space. In particular, we are interested in the asymptotic convergence of optimal  value functions of finite-state models to the optimal value of the original model and asymptotic optimality of policies obtained from finite-state models. Since we are only interested in asymptotic convergence guarantees, the conditions on the transition probability, the one-stage cost and the constraint functions are \emph{almost} strictly weaker than the conditions imposed in \cite{DuPr13,DuPr14-a,EsSuKuLy17}. For the approximation of the optimal discounted cost value function, we assume that the transition probability is continuous with respect to the Wasserstein distance of order $1$, and the one-stage cost and constraint functions are continuous. Moreover, unlike in \cite{DuPr13,DuPr14-a}, we also present a method to compute approximately optimal policies for the original problem via finite-state models, if
we replace continuity of the transition probability with respect to the Wasserstein distance of order $1$ with the continuity with respect to the total variation distance. For the approximation of the optimal average cost value function and optimal policy, we assume that the transition probability is continuous with respect to the total variation distance, the one-stage cost and constraint functions are continuous, and transition probability satisfies some drift and minorization conditions. Hence, except drift and minorization conditions, regularity conditions imposed for the average cost are strictly relaxed than the conditions imposed in \cite{EsSuKuLy17}. In addition, unlike in \cite{EsSuKuLy17}, we also establish the convergence of the optimal policies in the sense that for any $\varepsilon >0$, an $\varepsilon$-optimal policy can be constructively obtained through the solution of a finite model approximation.

In the second part of the paper, we derive upper bounds on the performance losses due to approximation in terms of the number of grid points that is used to discretize the state space. The conditions imposed in this part are similar to those in \cite{DuPr13,DuPr14-a,EsSuKuLy17}. However, since finite-models are obtained through different procedures in this paper and in \cite{DuPr13,DuPr14-a,EsSuKuLy17} and since we also provide a method to compute nearly optimal policies, the results in this part cannot be deduced from the results in
\cite{DuPr13,DuPr14-a,EsSuKuLy17}.

The only weakness of our approach compared to that of \cite{DuPr13,DuPr14-a} is the compactness assumption of the state space. In the future, we plan to extend these results to the constrained Markov decision processes with unbounded state spaces. Furthermore, we also plan to study approximation problem for the risk-sensitive Markov decision processes \cite{BaOt11,BaRi14,Ugu17,Ugu18}.

The rest of the paper is organized as follows. In Section~\ref{sec2} we introduce constrained Markov decision processes and construct the finite-state model. In Section~\ref{constrained-disc} we study the approximation problem for constrained MDPs with discounted cost criterion. In Section~\ref{constrained:ave} analogous approximation results are obtained for constrained MDPs with average cost criterion. In Section~\ref{constrained:rateconv} we derive upper bounds on the performance losses due to approximation for both discounted and average cost criteria. Section~\ref{conc} concludes the paper.

\smallskip

\noindent\textbf{Notation.} For a metric space $\sE$ equipped with its Borel $\sigma$-algebra $\B(\sE)$, let $B(\sE)$ and $C_{b}(\sE)$ denote the set of all bounded measurable and bounded continuous real functions, respectively. For any $u \in C_b(\sE)$ or $u \in B(\sE)$, let $\|u\| \coloneqq \sup_{x \in \sE} |u(x)|$ which turns $C_{b}(\sE)$ and $B(\sE)$ into Banach spaces. Let $\Lip(\sE)$ denote the set of all Lipschitz continuous functions on $\sE$ and $\Lip(\sE,K)$ denotes the set of all $g \in \Lip(\sE)$ with Lipschitz constant less than $K$. Let $\M(\sE)$, $\M_+(\sE)$, and $\P(\sE)$ denote the set of all signed, positive, and probability measures on $\sE$, respectively. For any $\nu \in \M(\sE)$ and measurable real function $g$ on $\sE$, define $\langle\nu,g\rangle \coloneqq \int g d\nu$. Given vectors $\balpha$ and $\bdelta$ in the Euclidean space $\R^{q}$, let $\langle\balpha,\bdelta\rangle$ denote the usual inner product. Let $\mathbf{1}$ and $\mathbf{0}$ be the elements of $\R^{q}$ with all components equal to $1$ and $0$, respectively. Unless otherwise specified, the term `measurable' will refer to Borel measurability.

\section{Constrained Markov Decision Processes}\label{sec2}

We consider a discrete-time constrained Markov decision process (MDP) with \emph{state space} $\sX$ and \emph{action space} $\sA$, where $\sX$ and $\sA$ are Borel spaces (i.e., Borel subsets of complete and separable metric spaces). We assume that the \emph{set of admissible actions} for any $x \in \sX$ is $\sA$. Let the \emph{stochastic kernel} $p(\,\cdot\,|x,a)$ denote the \emph{transition probability} of the next state given that previous state-action pair is $(x,a)$ \cite{HeLa96}. Hence, we have: (i) $p(\,\cdot\,|x,a)$ is a element of $\P(\sX)$ for all $(x,a)$, and (ii) $p(D|\,\cdot\,,\,\cdot\,)$ is a measurable function from $\sX\times\sA$ to $[0,1]$ for each $D\in\B(\sX)$. The \emph{one-stage cost} function $c$ is a measurable function from $\sX \times \sA$ to $\R_{+}$. The probability measure $\gamma \in \P(\sX)$ denotes the initial distribution. Therefore, the components
\begin{align}
\bigl(\sX,\sA,p,c,\gamma\bigr) \nonumber
\end{align}
define a usual unconstrained Markov decision process. The last two components are the vectors of costs $\bd = (d_1,\ldots,d_q): \sX \times \sA \rightarrow \R^q_+$ and constraints $\bk = (k_1,\ldots,k_q) \in \R^q_+$, that will be used to define the constraints of the problem.

Define the history spaces $\sH_0 = \sX$ and $\sH_{t}=(\sX\times\sA)^{t}\times\sX$, $t=1,2,\ldots$ endowed with their product Borel $\sigma$-algebras generated by $\B(\sX)$ and $\B(\sA)$. A
\emph{policy} is a sequence $\pi=\{\pi_{t}\}$ of stochastic kernels on $\sA$ given $\sH_{t}$. The set of all policies is denoted by $\Pi$. Let $\Phi$ denote the set of stochastic kernels $\varphi$ on $\sA$ given $\sX$, and let $\mathbb{F}$ denote the set of all measurable functions $f$ from $\sX$ to $\sA$. A \emph{randomized stationary} policy is a constant sequence $\pi=\{\pi_{t}\}$ of stochastic kernels on $\sA$ given $\sX$ such that $\pi_{t}(\,\cdot\,|x)=\varphi(\,\cdot\,|x)$ for all $t$ for some $\varphi \in \Phi$. A \emph{deterministic stationary} policy is a constant sequence of stochastic kernels $\pi=\{\pi_{t}\}$ on $\sA$ given $\sX$ such that $\pi_{t}(\,\cdot\,|x)=\delta_{f(x)}(\,\cdot\,)$ for all $t$ for some $f \in \mathbb{F}$, where $\delta_b(\,\cdot\,)$ denotes the Dirac delta measure at point $b$. The set of randomized and deterministic stationary policies is identified with the sets $\Phi$ and $\mathbb{F}$, respectively.

According to the Ionescu Tulcea theorem \cite{HeLa96}, an initial distribution $\gamma$ on $\sX$ and a policy $\pi$ define a unique probability measure $P_{\gamma}^{\pi}$ on $\sH_{\infty}=(\sX\times\sA)^{\infty}$. The expectation with respect to $P_{\gamma}^{\pi}$ is denoted by $\cE_{\gamma}^{\pi}$. If $\gamma=\delta_x$, we write $P_{x}^{\pi}$ and $\cE_{x}^{\pi}$ instead of $P_{\delta_x}^{\pi}$ and $\cE_{\delta_x}^{\pi}$.

For each policy $\pi \in \Pi$ and $\beta \in (0,1)$ consider the $\beta$-discounted cost functions
\begin{align}
J(\pi,\gamma) &= (1-\beta) \cE_{\gamma}^{\pi}\biggl[\sum_{t=0}^{\infty}\beta^{t}c(X_{t},A_{t})\biggr], \nonumber \\
J_l(\pi,\gamma) &= (1-\beta) \cE_{\gamma}^{\pi}\biggl[\sum_{t=0}^{\infty}\beta^{t}d_l(X_{t},A_{t})\biggr] \text{ for } l=1,\ldots,q. \nonumber
\end{align}
We normalize the usual discounted cost by the coefficient $(1-\beta)$ to simplify some technical details.
Note that, for the discounted cost, same discount factor $\beta$ is used to define the cost function $J$ and the constraint functions $J_l$. Similarly, for each policy $\pi \in \Pi$, consider the average cost functions
\begin{align}
V(\pi,\gamma) &= \limsup_{T\rightarrow\infty} \frac{1}{T} \cE_{\gamma}^{\pi}\biggl[\sum_{t=0}^{T-1} c(X_{t},A_{t})\biggr], \nonumber \\
V_l(\pi,\gamma) &= \limsup_{T\rightarrow\infty} \frac{1}{T} \cE_{\gamma}^{\pi}\biggl[\sum_{t=0}^{T-1} d_l(X_{t},A_{t})\biggr] \text{ for } l=1,\ldots,q. \nonumber
\end{align}
Using the above notation, the constrained decision problems for discounted and average cost criteria can be defined as follows:
\begin{align}
\textbf{(CP$^{\bk}$)} \text{                         }&\text{minimize}_{\pi \in \Pi} W(\pi,\gamma)
\nonumber \\*
&\text{subject to  } W_l(\pi,\gamma) \leq k_l \text{  } \text{for } l=1,\ldots,q,
\nonumber
\end{align}
where $W \in \{J,V\}$. Note that if $W=J$ then one must have $W_l=J_l$ for all $l$, and similarly if $W=V$ then one must have $W_l=V_l$ for all $l$. We refer the reader to the book \cite{Alt99} to see how these constraints could be introduced in real-life applications.

In this paper, we assume that the following conditions hold for both discounted cost and average cost criteria.
\begin{assumption}
\label{as0}
\begin{itemize}
\item [ ]
\item [(a)] $\sX$ and $\sA$ are compact.
\item [(b)] The one-stage cost function $c$ and the constraint functions $d_l$ ($l=1,\ldots,q$) are continuous.
\item [(c)] There exists a policy $\pi \in \Pi$ such that $W_l(\pi,\gamma) < k_l$ for $l=1,\ldots,q$, where $W_l \in \{J_l,V_l\}$. In other words, $(W_{1}(\pi,\gamma),\ldots,W_{q}(\pi,\gamma)) + \balpha' = \bk$ for some $\balpha' > \mathbf{0}$.
\end{itemize}
\end{assumption}

In this paper, the goal is to construct reduced finite-state models to compute approximate optimal values and optimal policies. To this end, we first introduce the finite-state models.

\subsection{Finite-State Model}
\label{sec3}

In this section, we describe the construction of the finite-state models, which is adopted from our earlier work \cite{SaYuLi17}. Let $d_{\sX}$ denote the metric on $\sX$. Since the state space $\sX$ is compact, there exists a sequence $\bigl(\{x_{n,i}\}_{i=1}^{k_n}\bigr)_{n\geq1}$ of finite subsets of $\sX$ such that for all $n$,
\begin{align}
\min_{i\in\{1,\ldots,k_n\}} d_{\sX}(x,x_{n,i}) < 1/n \text{ for all } x \in \sX. \label{eq4}
\end{align}
In general, the size of the required discretization $k_n$ to have (\ref{eq4}) scales with the dimension of the state space as can be seen in (\ref{quantcof}). Let $\sX_n \coloneqq \{x_{n,1},\ldots,x_{n,k_n}\}$ and define function $Q_n$ mapping $\sX$ to $\sX_n$ by
\begin{align}
Q_n(x) \coloneqq \argmin_{y \in \sX_n} d_{\sX}(x,y),\nonumber
\end{align}
where ties are broken so that $Q_n$ is measurable. The function $Q_n$ is often called a nearest neighbor quantizer with respect to distortion measure $d_{\sX}$ \cite{GrNe98}. For each $n$, $Q_n$ induces a partition $\{\S_{n,i}\}_{i=1}^{k_n}$ of the state space $\sX$ given by
\begin{align}
\S_{n,i} = \{x \in \sX: Q_n(x)=x_{n,i}\}, \nonumber
\end{align}
with diameter $\diam(\S_{n,i}) \coloneqq \sup_{(x,y)\in\S_{n,i}\times\S_{n,i}} d_{\sX}(x,y) < 2/n$. Let $\{\nu_n\}$ be a sequence of measures on $\sX$ satisfying
\begin{align}
\nu_n(\S_{n,i}) > 0 \text{  for all  } i,n.  \nonumber
\end{align}
Since $\nu_n(\S_{n,i}) > 0$ and $\S_{n,i} \in \B(\sX)$ for all $i$ and $n$, one can define probability measures $\nu_{n,i}$ on $\S_{n,i}$ by restricting $\nu_n$ to $\S_{n,i}$:
\begin{align}
\nu_{n,i}(\,\cdot\,) \coloneqq \frac{\nu_n(\,\cdot\,)}{\nu_n(\S_{n,i})}. \nonumber
\end{align}
The probability measures $\nu_{n,i}$ will be used to define a sequence of finite-state constrained MDPs, denoted as MDP$_{n}$ ($n\geq1$), to approximate the original model. For each $n$, define the transition probability $p_n$ on $\sX_n$ given $\sX_n\times\sA$, the one-stage cost function $c_n: \sX_n\times\sA \rightarrow \R_{+}$, and the functions $\bd_n=(d_{1,n},\ldots,d_{q,n}): \sX\times\sA \rightarrow \R^q_+$ by
\begin{align}
p_n(x_{n,j}|x_{n,i},a) &\coloneqq \int_{\S_{n,i}} Q_n \ast p(x_{n,j}|x,a) \nu_{n,i}(dx), \nonumber \\
c_n(x_{n,i},a) &\coloneqq \int_{\S_{n,i}} c(x,a) \nu_{n,i}(dx), \nonumber \\
d_{l,n}(x_{n,i},a) &\coloneqq \int_{\S_{n,i}} d_l(x,a) \nu_{n,i}(dx) \text{  for  } l=1,\ldots,q,\nonumber
\end{align}
where $Q_n\ast p(\,\cdot\,|x,a) \in \P(\sX_n)$ is the pushforward of the measure $p(\,\cdot\,|x,a)$ with respect to $Q_n$; that is,
\begin{align}
Q_n\ast p(x_{n,j}|x,a) = p\bigl(\{x \in \sX: Q_n(x) = x_{n,j} \}|x,a\bigr), \nonumber
\end{align}
for all $x_{n,j} \in \sX_n$. For each $n$, we define MDP$_n$ as a constrained Markov decision process with the following components: $\sX_n$ is the state space, $\sA$ is the action space, $p_n$ is the transition probability, $c_n$ is the one-stage cost function, $\gamma_{n} \coloneqq Q_n \ast \gamma$ is the initial distribution, $\bd_n$ is the function defining the constraints, and $\bk$ is the constraint vector. History spaces and policies are defined in a similar way as in the original model. Let $\Pi_n$, $\Phi_n$ and $\mathbb{F}_n$ denote the set of all, randomized stationary and deterministic stationary policies of MDP$_n$, respectively. For any $\varphi \in \Phi_n$, let $\overline{\varphi} \in \Phi$ denote its \emph{extension} to $\sX$, where $\overline{\varphi}$ is defined by
\begin{align}
\overline{\varphi}(\,\cdot\,|x) \coloneqq \varphi(\,\cdot\,|Q_n(x)). \nonumber
\end{align}
For each policy $\pi \in \Pi_n$, the $\beta$-discounted and average costs for the functions $c_n$ and $d_{l,n}$ ($l=1,\ldots,q$) are defined analogously and are denoted by $W_n$ and $W_{l,n}$, respectively, where $W \in \{J,V\}$. Then, as in the definition of (\textbf{CP}$^{\bk}$), we define the constrained problem for MDP$_n$ by

\begin{align}
\textbf{(CP$_n^{\bk}$)} \text{                         }&\text{minimize}_{\pi \in \Pi_n} W_n(\pi,\gamma_{n})
\nonumber \\*
&\text{subject to  } W_{l,n}(\pi,\gamma_{n}) \leq k_l \text{  } (l=1,\ldots,q),
\nonumber
\end{align}
where $W \in \{J,V\}$.

\section{Asymptotic Approximation of Discounted Cost Problems}\label{constrained-disc}

In this section, we first consider the approximation of the optimal value; that is, we show that the optimal value of the finite-state model converges to the optimal value of the original model. Then, we establish a method for computing near optimal policies using finite-state models for the original constrained Markov decision process.

It is important note that in the results that we are aware of (e.g., \cite{DuPr13,DuPr14-a,EsSuKuLy17}), dealing with the approximation of constrained MDPs, only the convergence of the optimal values are established. Here, we also establish the convergence (in terms of costs) of the optimal policies of the finite-state models to the optimal policy of the original one. Therefore, the approach is constructive in that approximately optimal
policies are obtained.

\subsection{Asymptotic Approximation of optimal value}\label{app_val_disc}

In this section we prove that the optimal value of \textbf{(CP$_n^{\bk}$)} converges to the optimal value of \textbf{(CP$^{\bk}$)}, i.e.,
\begin{align}
\inf \textbf{(CP$_n^{\bk}$)} \rightarrow \inf \textbf{(CP$^{\bk}$)} \label{neweq1}
\end{align}
as $n\rightarrow\infty$.
To prove (\ref{neweq1}), the below assumptions will be imposed. Additional assumptions will be made for the problem of computing near optimal policies in Section~\ref{app_pol_disc}.

\begin{assumption}
\label{as1}
We assume that Assumption~\ref{as0} holds. In addition, we assume
\begin{itemize}
\item [(a)] The stochastic kernel $p(\,\cdot\,|x,a)$ is weakly continuous in $(x,a)$.
\end{itemize}
\end{assumption}

Throughout the paper, we consider the following model as a running example to illustrate the conditions needed on the system dynamics to have the assumptions imposed in this paper. This model can arise for instance in inventory/production system with finite capacity, control of water reservoirs with finite capacity, and fisheries management problem \cite[Section 1.3]{Her89}.
\begin{example}\label{exm1}
In this example we consider the system given by
\begin{align}
x_{t+1}=F(x_{t},a_{t},v_{t}), \text{ } t=0,1,2,\ldots \nonumber
\end{align}
where $\sX \subset \R^n$ and  $\sA \subset \R^m$ are compact sets for some $n,m\geq1$. The noise process $\{v_{t}\}$ is a sequence of independent and identically distributed (i.i.d.) random vectors on $\sV \subset \R^p$ for some $p\geq1$. We assume that $F$ is continuous in $(x,a)$ and the one-stage cost function $c$ and the constraints functions $d_l$ $(l=1,\ldots,q)$ are continuous. Under these conditions, this model satisfies Assumption~\ref{as0}-(a),(b) and Assumption~\ref{as1}-(a). No assumptions are needed on the noise process (not even the existence of a density is required).
\end{example}

Before proving (\ref{neweq1}), we formulate both (\textbf{CP}$^{\bk}$) and (\textbf{CP$_n^{\bk}$}) as linear programs on appropriate linear spaces. The duals of these linear programs will play a key role in proving (\ref{neweq1}). We refer the reader to \cite{HeGo00} and \cite[Chapter 6]{HeLa96} for a linear programming formulation of constrained MDPs with discounted cost function.

Recall that for any metric space $\sE$, $\M(\sE)$ denotes the set of finite signed measures on $\sE$ and $B(\sE)$ denotes the set of bounded measurable real functions. Consider the vector spaces $\bigl(\M(\sX\times\sA),B(\sX\times\sA)\bigr)$ and $\bigl(\M(\sX),B(\sX)\bigr)$. Let us define bilinear forms on $\bigl(\M(\sX\times\sA),B(\sX\times\sA)\bigr)$ and on $\bigl(\M(\sX),B(\sX)\bigr)$ by letting
\begin{align}
\langle \zeta,v  \rangle &\coloneqq \int_{\sX\times\sA} v(x,a) \zeta(dx,da), \label{eqqq1} \\
\langle \mu,u  \rangle &\coloneqq \int_{\sX} u(x) \mu(dx) \label{eqqq2},
\end{align}
where $\zeta \in \M(\sX\times\sA)$, $v \in B(\sX\times\sA)$, $\mu \in \M(\sX)$, and $u \in B(\sX)$. The bilinear form in (\ref{eqqq1}) constitutes duality between $\M(\sX\times\sA)$ and $B(\sX\times\sA)$, and the bilinear form in (\ref{eqqq2}) constitutes duality between $\M(\sX)$ and $B(\sX)$ \cite[Chapter IV.3]{Bar02}. Hence, the topologies on these spaces should be understood as the weak topology of the duality induced by these bilinear forms. For any $\zeta \in \M(\sX\times\sA)$, let $\hat{\zeta} \in \M(\sX)$ denote the marginal of $\zeta$ on $\sX$, i.e.,
\begin{align}
\hat{\zeta}(\,\cdot\,)=\zeta(\,\cdot\,\times \sA). \nonumber
\end{align}
We define the linear maps $\T: \M(\sX\times\sA) \rightarrow \M(\sX)$ and $\L: \M(\sX\times\sA) \times \R^q \rightarrow \M(\sX)\times\R^q$ by
\begin{align}
\T\zeta(\,\cdot\,) &= \hat{\zeta}(\,\cdot\,) - \beta \int_{\sX\times\sA} p(\,\cdot\,|x,a) \zeta(dx,da) \nonumber \\
\L(\zeta,\balpha) &= \biggl(\T\zeta, \langle \zeta,\bd \rangle + \balpha \biggr), \nonumber
\end{align}
where $\langle \zeta,\bd \rangle \coloneqq \bigl(\langle \zeta, d_1 \rangle,\ldots,\langle \zeta,d_q \rangle\bigr)$. Then, (\textbf{CP}$^{\bk}$) is equivalent to the following equality constrained linear program \cite[Lemma 3.3 and Section 4]{HeGo00}, which is also denoted by (\textbf{CP}$^{\bk}$):
\begin{align}
(\textbf{CP}^{\bk}) \text{                         }&\text{minimize}_{(\zeta,\balpha) \in \M_+(\sX\times\sA)\times\R^q_+} \text{ } \langle \zeta,c \rangle
\nonumber \\*
&\text{subject to  } \L(\zeta,\balpha) = \bigl((1-\beta)\gamma,\bk\bigr). \label{aaaaa}
\end{align}

\begin{remark} For any policy $\pi \in \Pi$, define the $\beta$-discount expected occupation measure as
\begin{align}
\zeta^{\pi}(C) \coloneqq (1-\beta) \sum_{t=0}^{\infty} \beta^t P_{\gamma}^{\pi} \biggl[ (X_t,A_t) \in C \biggr], \text{ } C \in \B(\sX\times\sA). \nonumber
\end{align}
Note that $\zeta^{\pi}$ is a probability measure on $\sX \times \sA$ as a result of the normalizing constant $(1-\beta)$. In the absence of this normalization, we would have to deal with non-probability measures which complicates the analysis. One can prove that $\zeta^{\pi}$ satisfies
\begin{align}
\hat{\zeta}^{\pi}(\,\cdot\,) = (1-\beta) \gamma(\,\cdot\,) + \beta \int_{\sX \times \sA} p(\,\cdot\,|x,a) \zeta^{\pi}(dx,da). \label{occup}
\end{align}
Conversely, if any finite measure $\zeta$ satisfies (\ref{occup}), then it is a $\beta$-discount expected occupation measure of some policy $\pi \in \Pi$ \cite[Lemma 3.3]{HeGo00}. Using the $\beta$-discount expected occupation measure, we can write
\begin{align}
J(\pi,\gamma) = \langle \zeta^{\pi}, c \rangle \text{ and } J_l(\pi,\gamma) = \langle \zeta^{\pi}, d_l \rangle. \nonumber
\end{align}
Therefore, (\textbf{CP}$^{\bk}$) can be written in the following alternative form:
\begin{align}
\textbf{(CP$^{\bk}$)} \text{                         }&\text{minimize}_{\zeta \in \M_+(\sX\times\sA)} \text{ } \langle \zeta,c \rangle
\nonumber \\*
&\text{subject to  }  \langle \zeta, d_l \rangle \leq k_l \text{ for } l=1,\ldots,q \nonumber \\*
& \hat{\zeta}(\,\cdot\,) = (1-\beta) \gamma(\,\cdot\,) + \beta \int_{\sX \times \sA} p(\,\cdot\,|x,a) \zeta(dx,da). \nonumber
\end{align}
From this alternative formulation, it straightforward to obtain (\ref{aaaaa}).\phantom{x}$\square$
\end{remark}

\noindent Note that the adjoint $\T^*: B(\sX) \rightarrow B(\sX\times\sA)$ of $\T$ and the adjoint $\L^*: B(\sX)\times\R^q \rightarrow B(\sX\times\sA)\times\R^q$ of $\L$ are given by (see \cite[Section 4]{HeGo00})
\begin{align}
{\T}^*u(x,a) &= u(x) - \beta \int_{\sX} u(y) p(dy|x,a), \nonumber \\
{\L}^*(u,\bdelta) &= \biggl({\T}^*u + \sum_{l=1}^q \delta_l d_l, \bdelta\biggr). \nonumber
\end{align}
Hence, the dual (\textbf{CP$^{*,\bk}$}) of (\textbf{CP}$^{\bk}$) is
\begin{align}
\textbf{(CP$^{*,\bk}$)} \text{                         }&\text{maximize}_{(u,\bdelta) \in B(\sX)\times\R^q} \text{ } (1-\beta)\langle \gamma,u \rangle + \langle \bk,\bdelta \rangle
\nonumber \\*
&\text{subject to  } {\L}^*(u,\bdelta) \leq (c,\mathbf{0}). \nonumber
\end{align}
Here, ${\L}^*(u,\bdelta) \leq (c,\mathbf{0})$ can be written more explicitly as
\begin{align}
u(x) \leq c(x,a) - \sum_{l=1}^q \delta_l d_l(x,a) + \beta \int_{\sX} u(y) p(dy|x,a) \nonumber
\end{align}
for all $(x,a) \in \sX\times\sA$, where $\delta_l \leq 0$ for $l=1,\ldots,q$.

By replacing $(\sX,p,c,\bd,\gamma)$ with $(\sX_n,p_n,c_n,\bd_n,\gamma_n)$ above, we can write the equivalent equality constraint linear program for (\textbf{CP$_n^{\bk}$}) as follows:
\begin{align}
\textbf{(CP$_n^{\bk}$)} \text{                         }&\text{minimize}_{(\zeta,\balpha) \in \M_+(\sX_n\times\sA)\times\R^q_+} \text{ } \langle \zeta,c_n \rangle
\nonumber \\*
&\text{subject to  } {\L}_n(\zeta,\balpha) = \bigl((1-\beta)\gamma_n,\bk\bigr), \nonumber
\end{align}
where $\T_n: \M(\sX_n\times\sA) \rightarrow \M(\sX_n)$ and $\L_n: \M(\sX_n\times\sA) \times \R^q \rightarrow \M(\sX_n)\times\R^q$ are given by
\begin{align}
{\T}_n\zeta(\,\cdot\,) &= \hat{\zeta}(\,\cdot\,) - \beta \sum_{i=1}^{k_n}\int_{\sA} p_n(\,\cdot\,|x_{n,i},a) \zeta(x_{n,i},da) \nonumber \\
{\L}_n(\zeta,\balpha) &= \biggl({\T}_n\zeta, \langle \zeta,\bd_n \rangle + \balpha \biggr). \nonumber
\end{align}
Similarly, the dual (\textbf{CP$_n^{*,\bk}$}) of (\textbf{CP$_n^{\bk}$}) is given by
\begin{align}
\textbf{(CP$_n^{*,\bk}$)} \text{                         }&\text{maximize}_{(u,\bdelta) \in B(\sX_n)\times\R^q} \text{ }  (1-\beta)\langle \gamma_n,u \rangle + \langle \bk,\bdelta \rangle
\nonumber \\*
&\text{subject to  } {\L}_n^*(u,\bdelta) \leq (c_n,\mathbf{0}), \nonumber
\end{align}
where the adjoint ${\T}_n^*: B(\sX_n) \rightarrow B(\sX_n\times\sA)$ of $\T_n$ and the adjoint $\L_n^*: B(\sX_n)\times\R^q \rightarrow B(\sX_n\times\sA)\times\R^q$ of $\L_n$ are given by
\begin{align}
{\T}_n^*u(x_{n,i},a) &= u(x_{n,i}) - \beta \sum_{j=1}^{k_n} u(x_{n,j}) p_n(x_{n,j}|x_{n,i},a), \nonumber \\
{\L}_n^*(u,\bdelta) &= \biggl({\T}_n^*u + \sum_{l=1}^q \delta_l d_{l,n}, \bdelta\biggr). \nonumber
\end{align}
Here, ${\L}_n^*(u,\bdelta) \leq (c_n,\mathbf{0})$ can be written more explicitly as
\begin{align}
&u(x_{n,i}) \leq c_n(x_{n,i},a) - \sum_{l=1}^q \delta_l d_{l,n}(x_{n,i},a) \nonumber \\
&\phantom{xxxxxxxxxxxxxxxxxx}+ \beta \sum_{j=1}^{k_n} u(x_{n,j}) p_n(x_{n,j}|x_{n,i},a) \nonumber
\end{align}
for all $(x_{n,i},a) \in \sX_n\times\sA$, where $\delta_l \leq 0$ for $l=1,\ldots,q$.

If Assumption~\ref{as1} holds, then by \cite[Theorems 3.2 and 4.3]{HeGo00} we have
\begin{align}
\sup (\textbf{CP$^{*,\bk}$}) &= \min (\textbf{CP}^{\bk}) \nonumber \\
\sup (\textbf{CP$_n^{*,\bk}$}) &= \min (\textbf{CP$_n^{\bk}$}), \nonumber
\end{align}
where the $``\min"$ notation signifies that there exist optimal policies for (\textbf{CP}$^{\bk}$) and (\textbf{CP}$_n^{\bk}$). Furthermore, if $(\zeta^*,\balpha) \in \M_+(\sX \times \sA) \times \R_+^q$ and $(\zeta_n^*,\balpha_n) \in \M_+(\sX_n \times \sA) \times \R_+^q$ are minimizers for (\textbf{CP}$^{\bk}$) and (\textbf{CP$_n^{\bk}$}), respectively, then the optimal (randomized stationary) policies $\varphi^* \in \Phi$ and $\varphi_n^* \in \Phi_n$ for MDP and MDP$_n$ are given by disintegrating $\zeta^*$ and $\zeta_n^*$ as (see \cite[Theorem 6.3.7]{HeLa96})
\begin{align}
\zeta^*(dx,da) &= \varphi^*(da|x) \hat{\zeta}^*(dx), \nonumber \\
\zeta_n^*(dx,da) &= \varphi_n^*(da|x) \hat{\zeta_n}^*(dx). \nonumber
\end{align}

The following theorem is the main result of this section.

\begin{theorem}\label{main1}
We have
\begin{align}
\lim_{n\rightarrow \infty} \bigl|\min (\textbf{CP$_n^{\bk}$}) - \min (\textbf{CP}^{\bk})\bigr| = 0; \label{neweq2}
\end{align}
that is, the optimal value of constrained MDP$_n$ converges to the optimal value of constrained MDP as $n\rightarrow\infty$.
\end{theorem}

To prove Theorem~\ref{main1},  for each $n\geq1$, we introduce another constrained MDP, denoted by $\overline{\text{MDP}}_n$, with the components
\begin{align}
\bigl(\sX,\sA,q_n,b_n,\br_n,\gamma\bigr), \nonumber
\end{align}
where $q_n: \sX \times \sA \rightarrow \P(\sX)$, $b_n : \sX \times \sA \rightarrow \R_{+}$, and $\br_n = (r_{1,n},\ldots,r_{q,n}): \sX \times \sA \rightarrow \R_{+}^q$ are defined as
\begin{align}
q_n(\,\cdot\,|x,a) &= \int_{\S_{n,i_n(x)}} p(\,\cdot\,|z,a) \nu_{n,i_n(x)}(dz), \nonumber \\
b_n(x,a) &= \int_{\S_{n,i_n(x)}} c(z,a) \nu_{n,i_n(x)}(dz), \nonumber \\
r_{l,n}(x,a) &= \int_{\S_{n,i_n(x)}} d_l(z,a) \nu_{n,i_n(x)}(dz), \nonumber
\end{align}
where $i_n: \sX \rightarrow \{1,\ldots,k_n\}$ maps $x$ to the index of the quantization region it belongs to. As before, the constrained decision problem that corresponds to $\overline{\text{MDP}}_n$ can be formulated as an equality constrained linear program given by
\begin{align}
(\overline{\textbf{CP}}_n^{\bk}) \text{                         }&\text{minimize}_{(\zeta,\balpha) \in \M_+(\sX\times\sA)\times\R^q_+} \text{ } \langle \zeta,b_n \rangle
\nonumber \\*
&\text{subject to  } \overline{{\L}}_n(\zeta,\balpha) = \bigl((1-\beta)\gamma,\bk\bigr), \nonumber
\end{align}
where $\overline{\T}_n: \M(\sX\times\sA) \rightarrow \M(\sX)$ and $\overline{\L}_n: \M(\sX\times\sA) \times \R^q \rightarrow \M(\sX)\times\R^q$ are given by
\begin{align}
{\overline{\T}}_n\zeta(\,\cdot\,) &= \hat{\zeta}(\,\cdot\,) - \beta \int_{\sX\times\sA} q_n(\,\cdot\,|x,a) \zeta(dx,da) \nonumber \\
{\overline{\L}}_n(\zeta,\balpha) &= \biggl({\overline{\T}}_n\zeta, \langle \zeta,\br_n \rangle + \balpha \biggr). \nonumber
\end{align}
Furthermore, the dual ($\overline{\textbf{CP}}_n^{*,\bk}$) of ($\overline{\textbf{CP}}_n^{\bk}$) is given by
\begin{align}
(\overline{\textbf{CP}}_n^{*,\bk}) \text{                         }&\text{maximize}_{(u,\bdelta) \in B(\sX)\times\R^q} \text{ } (1-\beta)\langle \gamma,u \rangle + \langle \bk,\bdelta \rangle
\nonumber \\*
&\text{subject to  } {\L}_n^*(u,\bdelta) \leq (b_n,\mathbf{0}), \nonumber
\end{align}
where the adjoint ${\overline{\T}}_n^*: B(\sX) \rightarrow B(\sX\times\sA)$ of $\overline{\T}_n$ and the adjoint $\overline{\L}_n^*: B(\sX)\times\R^q \rightarrow B(\sX\times\sA)\times\R^q$ of $\overline{\L}_n$ are as follows
\begin{align}
{\overline{\T}}_n^*u(x,a) &= u(x) - \beta \int_{\sX} u(y) q_n(dy|x,a), \nonumber \\
{\overline{\L}}_n^*(u,\bdelta) &= \biggl({\overline{\T}}_n^*u + \sum_{l=1}^q \delta_l r_{l,n}, \bdelta\biggr). \nonumber
\end{align}
Here, ${\overline{\L}}_n^*(u,\bdelta) \leq (b_n,\mathbf{0})$ can be written more explicitly as
\begin{align}
u(x) \leq b_n(x,a) - \sum_{l=1}^q \delta_l r_{l,n}(x,a) + \beta \int_{\sX} u(y) q_n(dy|x,a) \nonumber
\end{align}
for all $(x,a) \in \sX\times\sA$, where $\delta_l \leq 0$ for $l=1,\ldots,q$.

For each $\bdelta \in \R^q$, define $\Gamma^{\bdelta}: B(\sX) \rightarrow B(\sX)$, $\Gamma_n^{\bdelta}: B(\sX_n) \rightarrow B(\sX_n)$, and $\overline{\Gamma}_n^{\bdelta}: B(\sX) \rightarrow B(\sX)$ by
\begin{align}
&\Gamma^{\bdelta} u(x) = \min_{a \in \sA} \biggl[ c_{\bdelta}(x,a) + \beta \int_{\sX} u(y) p(dy|x,a) \biggr] \nonumber \\
&\Gamma_n^{\bdelta}u(x_{n,i}) = \min_{a \in \sA} \int \biggl[ c_{\bdelta}(z,a) + \beta \int_{\sX} \hat{u}(y) p(dy|x,a) \biggr] \nu_{n,i}(dz) \nonumber \\
&\overline{\Gamma}_n^{\bdelta}u(x) \nonumber \\
&\phantom{xxx}= \min_{a \in \sA} \int \biggl[ c_{\bdelta}(z,a) + \beta \int_{\sX} u(y) p(dy|x,a) \biggr] \nu_{n,i_n(x)}(dz), \nonumber
\end{align}
where $c_{\bdelta}(z,a) = c(z,a) - \sum_{l=1}^q \delta_l d_l(z,a)$ and $\hat{u}=u\circ Q_n$ for $u\in B(\sX_n)$. Here, $\Gamma^{\bdelta}$ is the Bellman optimality operator for the unconstrained Markov decision process with the components $\bigl(\sX,\sA,p,c-\sum_{l=1}^q\delta_l d_{l},\gamma\bigr)$, $\Gamma_n^{\bdelta}$ is the Bellman optimality operator of the unconstrained Markov decision process with components $\bigl(\sX_n,\sA,p_n,c_n-\sum_{l=1}^q\delta_l d_{l,n},\gamma_n\bigr)$, and $\overline{\Gamma}_n^{\bdelta}$ is the Bellman optimality operator of the unconstrained Markov decision process with components $\bigl(\sX,\sA,q_n,b_n-\sum_{l=1}^q\delta_l r_{l,n},\gamma\bigr)$. They are both contractions with modulus $\beta$. Hence, they have fixed points $u^*_{\bdelta}$, $u^*_{n,\bdelta}$, and $\overline{u}^*_{n,\bdelta}$, respectively. Furthermore, the fixed point $\overline{u}^*_{n,\bdelta}$ of $\overline{\Gamma}_n^{\bdelta}$ is the piecewise constant extension of the fixed point $u^*_{n,\bdelta}$ of $\Gamma_n^{\bdelta}$, i.e.,
\begin{align}
\overline{u}^*_{n,\bdelta} = u^*_{n,\bdelta}\circ Q_n. \nonumber
\end{align}
Using these operators one can rewrite (\textbf{CP}$^{*,\bk}$), (\textbf{CP}$_n^{*,\bk}$), and ($\overline{\textbf{CP}}_n^{*,\bk}$) as
\begin{align}
(\textbf{CP}^{*,\bk}) \text{                         }&\text{maximize}_{(u,\bdelta) \in B(\sX) \times \R^q_-} \text{ } (1-\beta)\langle \gamma,u \rangle + \langle \bk,\bdelta \rangle
\nonumber \\*
&\text{subject to  } u \leq \Gamma^{\bdelta}u, \nonumber
\end{align}
\begin{align}
(\textbf{CP}_n^{*,\bk}) \text{                         }&\text{maximize}_{(u,\bdelta) \in B(\sX_n) \times \R^q_-} \text{ } (1-\beta)\langle \gamma_n,u \rangle + \langle \bk,\bdelta \rangle
\nonumber \\*
&\text{subject to  } u \leq \Gamma_n^{\bdelta}u, \nonumber
\end{align}
\begin{align}
(\overline{\textbf{CP}}_n^{*,\bk}) \text{                         }&\text{maximize}_{(u,\bdelta) \in B(\sX) \times \R^q_-} \text{ } (1-\beta)\langle \gamma,u \rangle + \langle \bk,\bdelta \rangle
\nonumber \\*
&\text{subject to  }  u \leq \overline{\Gamma}_n^{\bdelta}u. \nonumber
\end{align}
Observe that if $u \leq \Gamma^{\bdelta} u$, then $u \leq u^*_{\bdelta}$. Indeed, for any $m\geq1$, we have $u \leq \bigl(\Gamma^{\bdelta}\bigr)^m u$ since $\Gamma^{\bdelta}$ is monotone $\bigl($i.e., $u \leq v$ implies $\Gamma^{\bdelta}u \leq \Gamma^{\bdelta}v$$\bigr)$. Furthermore, $(\Gamma^{\bdelta}\bigr)^m u$ converges to $u^*_{\bdelta}$ by the Banach fixed point theorem. Hence, $u \leq u^*_{\bdelta}$. The same conclusions can be made for $u^*_{n,\bdelta}$ and $\overline{u}^*_{n,\bdelta}$. Thus, we can write
\begin{align}
(\textbf{CP}^{*,\bk}) \text{                         }&\text{maximize}_{\bdelta \in \R^q_-} (1-\beta)\langle \gamma,u^*_{\bdelta} \rangle + \langle \bk,\bdelta \rangle, \nonumber  \\
(\textbf{CP}_n^{*,\bk}) \text{                         }&\text{maximize}_{\bdelta \in \R^q_-} (1-\beta)\langle \gamma_n,u^*_{n,\bdelta} \rangle + \langle \bk,\bdelta \rangle, \nonumber \\
(\overline{\textbf{CP}}_n^{*,\bk}) \text{                         }&\text{maximize}_{\bdelta \in \R^q_-} (1-\beta)\langle \gamma,\overline{u}^*_{n,\bdelta} \rangle + \langle \bk,\bdelta \rangle. \nonumber
\end{align}
The following result states that MDP$_n$ and $\overline{\text{MDP}}_n$ are essentially identical.

\begin{lemma}\label{lemma1}
We have
\begin{align}
\min (\textbf{CP}_n^{\bk}) = \min (\overline{\textbf{CP}}_n^{\bk}), \label{neweq3}
\end{align}
and if the randomized stationary policy $\varphi^* \in \Phi_n$ is optimal for $($\textbf{CP}$_n^{\bk})$, then its extension $\overline{\varphi}^*$ to $\sX$ is also optimal for $(\overline{\textbf{CP}}_n^{\bk})$ with the same cost function.
\end{lemma}

The proof of Lemma~\ref{lemma1} is given in Appendix~\ref{prooflemma1}. Lemma~\ref{lemma1} implies that to prove Theorem~\ref{main1}, it is sufficient to show that
\begin{align}
\lim_{n\rightarrow\infty} \bigl| \min (\overline{\textbf{CP}}_n^{\bk}) - \min (\textbf{CP}^{\bk}) \bigr| = 0. \nonumber
\end{align}
We use this fact in the proof of Theorem~\ref{main1}.

Let us define functions $G_n: \R^q_- \rightarrow \R$ and $G: \R^q_- \rightarrow \R$ by
\begin{align}
G_n(\bdelta) &= (1-\beta)\langle \gamma,\overline{u}^*_{n,\bdelta} \rangle + \langle \bk,\bdelta \rangle \nonumber \\
G(\bdelta) &= (1-\beta)\langle \gamma,u^*_{\bdelta} \rangle + \langle \bk,\bdelta \rangle. \nonumber
\end{align}
Hence,
\begin{align}
\sup(\overline{\textbf{CP}}_n^{*,\bk}) = \sup_{\bdelta \in \R^q_-} G_n(\bdelta) \text{ and } \sup(\textbf{CP}^{*,\bk}) = \sup_{\bdelta \in \R^q_-} G(\bdelta). \nonumber
\end{align}
By \cite[Theorem 2.4]{SaYuLi17}, we have $\lim_{n\rightarrow\infty}\| \overline{u}^*_{n,\bdelta} - u^*_{\bdelta} \|=0$ for all $\bdelta \in \R^q_-$. This implies that
\begin{align}
G_n(\bdelta) \rightarrow G(\bdelta) \label{neweq10}
\end{align}
as $n\rightarrow \infty$ for all $\bdelta \in \R^q_-$; that is, $G_n$ converges to $G$ pointwise. To prove Theorem~\ref{main1}, we need two technical results. The proof of the first result is given in the Appendix~\ref{prooflemma2}. The second result can be deduced from \cite[Theorems 3.6 and 4.10]{DuPr13}. Indeed, we state a similar result for the average cost in Lemma~\ref{lemma6} whose proof follows the arguments in \cite[Theorems 3.6 and 4.10]{DuPr13}. Therefore, an interested reader can analyze the proof of Lemma~\ref{lemma6} in the Appendix~\ref{prooflemma6} and see how it can be modified to obtain the second result.

\begin{lemma}\label{lemma2}
There exists $n(\bk)\geq1$ such that for each $n \geq n(\bk)$, one can find $(\zeta_n,\balpha_n) \in \M_+(\sX\times\sA)\times\R^q_+$ feasible for $(\overline{\textbf{CP}}_n^{\bk})$ and $\balpha_n \geq \balpha'/2$, where $\balpha'$ is the vector in Assumption~\ref{as0}-(c).
\end{lemma}

\begin{proposition}\label{prop1}
There exist $\bdelta_n^* \in \R^q_-$ $\bigr(n\geq n(\bk)\bigl)$ and $\bdelta^* \in \R^q_-$ such that
\begin{align}
G_n(\bdelta_n^*) &= \sup (\overline{\textbf{CP}}^{*,\bk}_n), \nonumber \\
G(\bdelta^*) &= \sup (\textbf{CP}^{*,\bk}), \nonumber
\end{align}
and $\|\bdelta_n^*\|_1, \|\bdelta^*\|_1 \leq K < \infty$, where $K \coloneqq \frac{ 2 \|c\|}{\alpha}$ and $\|\bdelta\|_1 \coloneqq \sum_{l=1}^q |\delta_l|$.
\end{proposition}

\begin{proof}[Proof of Theorem~\ref{main1}]
Recall that to prove Theorem~\ref{main1}, it is sufficient to prove
\begin{align}
\lim_{n\rightarrow\infty} \bigl | \min (\overline{\textbf{CP}}_n^{\bk}) - \min (\textbf{CP}^{\bk}) \bigr| &= 0, \nonumber \\
\intertext{or equivalently}
\lim_{n\rightarrow\infty} \bigl | \max (\overline{\textbf{CP}}_n^{*,\bk}) - \max (\textbf{CP}^{*,\bk}) \bigr| &= 0 \nonumber
\end{align}
since there is no duality gap. By Proposition~\ref{prop1}, the latter equation can be written as
\begin{align}
\lim_{n\rightarrow\infty} \bigl| \sup_{\bdelta \in \K} G_n(\bdelta) - \sup_{\bdelta \in \K} G(\bdelta) \bigr| = 0, \label{neweq6}
\end{align}
where $\K = \{\bdelta \in \R^q_-: \| \bdelta \|_1 \leq K\}$ is a compact subset of $\R^q_-$. Hence, if we can show that $G_n$ converges to $G$ uniformly on $\K$, then the proof is complete. We prove this by showing the relative compactness of $\{G_n\}$ with respect to the topology of uniform convergence.

First, we note that $\{G_n\}$ is equicontinuous with respect to the metric induced by the norm $\|\,\cdot\,\|_1$. Indeed, for any $\bdelta, \bdelta'$, we have
\begin{align}
&|G_n(\bdelta) - G_n(\bdelta')| \nonumber \\
&= \bigl| (1-\beta) \langle \gamma,\overline{u}^*_{n,\bdelta} \rangle + \langle \bk, \bdelta \rangle - (1-\beta) \langle \gamma,\overline{u}^*_{n,\bdelta'} \rangle - \langle \bk, \bdelta' \rangle \bigr| \nonumber \\
&\leq (1-\beta) \bigl| \langle \gamma,\overline{u}^*_{n,\bdelta} \rangle  - \langle \gamma,\overline{u}^*_{n,\bdelta'} \rangle \bigr| + \bigl| \langle \bk, \bdelta \rangle - \langle \bk, \bdelta' \rangle \bigr| \nonumber \\
&\leq (1-\beta) \biggl| \inf_{\zeta \in \P_n} \langle \zeta, b_n - \sum_{l=1}^q \delta_l r_{l,n} \rangle \nonumber \\
&\phantom{xxxxxxxxxxx}- \inf_{\zeta \in \P_n} \langle \zeta, b_n - \sum_{l=1}^q \delta_l' r_{l,n} \rangle \biggr| + \sum_{l=1}^q k_l |\delta_l - \delta_l'| \nonumber \\
&\leq (1-\beta) \sup_{\zeta \in \P_n} \biggl| \langle \zeta, \sum_{l=1}^q (\delta_l-\delta_l') r_{l,n} \rangle \biggr| + \max_{l=1,\ldots,q} k_l \text{  } \|\bdelta - \bdelta'\|_1 \nonumber \\
&\leq \biggl( (1-\beta) \max_{l=1,\ldots,q} \|r_{l,n}\| + \max_{l=1,\ldots,q} k_l \biggr) \text{  } \|\bdelta - \bdelta'\|_1, \nonumber
\end{align}
where $\P_n$ denotes the set of $\beta$-discount expected occupation measures for $\overline{\text{MDP}}_n$ (see the proof of Lemma~\ref{lemma2}). Since $\|r_{l,n}\| \leq \|d_l\|$, we have $G_n \in \Lip(\K,M)$ for all $n$, where
\begin{align}
M = \bigl( (1-\beta) \max_{l=1,\ldots,q} \|d_{l}\| + \max_{l=1,\ldots,q} k_l \bigr). \nonumber
\end{align}
Hence, $\{G_n\}$ is equicontinuous. Furthermore, it is also straightforward to prove that for any $\bdelta \in \K$, $\{G_n(\bdelta)\}$ is bounded. Thus, by Arzela-Ascoli theorem, $\{G_n\}$ is relatively compact with respect to the topology of uniform convergence. Recall that $G_n \rightarrow G$ pointwise $\bigl($see (\ref{neweq10})$\bigr)$, and therefore, every uniformly convergent subsequence of $\{G_n\}$ must converge to $G$. Together with the relative compactness of $\{G_n\}$, this implies that $G_n$ converges to $G$ uniformly.
\end{proof}

\subsection{Asymptotic Approximation of Optimal Policy}\label{app_pol_disc}

In this section we establish a method for computing near optimal policies using finite-state models for the constrained Markov decision problem (\textbf{CP}$^{\bk}$). To this end, we need to slightly strengthen Assumption~\ref{as1} by replacing Assumption~\ref{as1}-(a) with the continuity of $p(\,\cdot\,|x,a)$ in $(x,a)$ with respect to the total variation distance. In this section, we assume that the following assumptions hold:

\begin{assumption}\label{as2}
We suppose that Assumption~\ref{as0} holds. Furthermore, we assume
\begin{itemize}
\item [(a)] \text{ } The stochastic kernel $p(\,\cdot\,|x,a)$ is continuous in $(x,a)$ with respect to total variation distance.
\end{itemize}
\end{assumption}

\begin{example}\label{exm2}
Recall the model in Example~\ref{exm1}. In addition to the conditions in Example~\ref{exm1}, we assume that the noise admits a density with respect to the Lebesgue measure and this density is continuous. Then, one can prove that Assumption~\ref{as2}-(a) holds via Scheffe's Theorem (see, e.g., \cite[Theorem 16.12]{Bil95}).
\end{example}


\begin{remark}
In the rest of this paper, when we take the integral of any function with respect to $\nu_{n,i_n(x)}$, it is tacitly assumed that the integral is taken over all set $\S_{n,i_n(x)}$. Hence, we can drop $\S_{n,i_n(x)}$ in the integral for the ease of notation.
\end{remark}

For any $g \in B(\sX\times\sA)$ and any $\pi \in \Pi$, define $J^g(\pi,\gamma)$ and $J^g_n(\pi,\gamma)$ as the $\beta$-discounted costs of MDP and $\overline{\text{MDP}}_n$, respectively, when the one-stage cost function is $g$.
For each $n$, let the randomized stationary policy $\overline{\varphi}_n \in \Phi$ be the extension of a policy $\varphi_n \in \Phi_n$ to $\sX$. If we apply $\overline{\varphi}_n$ both to MDP and $\overline{\text{MDP}}_n$, we obtain two Markov chains, describing the state processes, with the following transition probabilities
\begin{align}
P_n(\,\cdot\,|x) &\coloneqq p(\,\cdot\,|x,\overline{\varphi}_n) = \int_{\sA} p(\,\cdot\,|x,a) \overline{\varphi}_n(da|x), \nonumber \\
R_n(\,\cdot\,|x) &\coloneqq q_n(\,\cdot\,|x,\overline{\varphi}_n) = \int_{\sA} q_n(\,\cdot\,|x,a) \overline{\varphi}_n(da|x). \nonumber
\end{align}
Furthermore, we can write $R_n(\,\cdot\,|x)$ as
\begin{align}
R_n(\,\cdot\,|x) = \int P_n(\,\cdot\,|z) \nu_{n,i_n(x)}(dz). \nonumber
\end{align}
For any $t\geq0$, we write $P_n^t(\,\cdot\,|\gamma)$ and $R_n^t(\,\cdot\,|\gamma)$ to denote the $t$-step transition probability of the Markov chains given the initial distribution $\gamma$; that is, $P_n^0(\,\cdot\,|\gamma)=R_n^0(\,\cdot\,|\gamma)=\gamma$ and for $t\geq1$
\begin{align}
P_n^{t+1}(\,\cdot\,|\gamma) &= \int_{\sX} P_n(\,\cdot\,|x) P_n^t(dx|\gamma), \nonumber \\
R_n^{t+1}(\,\cdot\,|\gamma) &= \int_{\sX} R_n(\,\cdot\,|x) R_n^t(dx|\gamma). \nonumber
\end{align}

Before stating the next lemma, we need some new notation. For any $g: \sX\times\sA \rightarrow \R$ and $n\geq1$, let
\begin{align}
g_{\overline{\varphi}_n}(x) &\coloneqq \int_{\sA} g(x,a) \overline{\varphi}_n(da|x) \nonumber \\
\intertext{and}
g_n(x,a) &\coloneqq \int g(z,a) \nu_{n,i_n(x)}(dz). \nonumber
\end{align}
Therefore, we can define
\begin{align}
g_{n,\overline{\varphi}_n}(x) = \int_{\sA} \int g(z,a) \nu_{n,i_n(x)}(dz) \overline{\varphi}_n(da|x). \nonumber
\end{align}

\begin{lemma}\label{lemma3}
Let $\{\varphi_n\}$ be a sequence such that $\varphi_n \in \Phi_n$ for all $n$. Then, for any $g \in \C_b(\sX\times\sA)$ and for any $t\geq1$, we have
\begin{align}
\lim_{n\rightarrow\infty} \biggl| \int_{\sX} g_{n,\overline{\varphi}_n}(x) R_n^t(dx|\gamma) - \int_{\sX} g_{\overline{\varphi}_n}(x) P_n^t(dx|\gamma) \biggr| = 0. \nonumber
\end{align}
\end{lemma}

The proof of Lemma~\ref{lemma3} is given in Appendix~\ref{prooflemma3}. Using Lemma~\ref{lemma3} we now prove the following result.

\begin{proposition}\label{prop2}
Let $\{\varphi_n\}$ be a sequence such that $\varphi_n \in \Phi_n$ for all $n$. For any $g \in C_b(\sX\times\sA)$, we have
\begin{align}
\lim_{n\rightarrow\infty} \bigl| J^{g_n}_n(\overline{\varphi}_n,\gamma) - J^g(\overline{\varphi}_n,\gamma) \bigr| = 0. \nonumber
\end{align}
\end{proposition}

\begin{proof}
We have
\begin{align}
&\limsup_{n\rightarrow\infty} \bigl| J^{g_n}_n(\overline{\varphi}_n,\gamma) - J^g(\overline{\varphi}_n,\gamma) \bigr| \nonumber \\
&= (1-\beta) \limsup_{n\rightarrow\infty}
\biggl| \sum_{t=0}^{\infty} \beta^t \int_{\sX} g_{n,\overline{\varphi}_n}(x) R_n^t(dx|\gamma) \nonumber \\
&\phantom{xxxxxxxxxxxxxxxxx}- \sum_{t=0}^{\infty} \beta^t \int_{\sX} g_{\overline{\varphi}_n}(x) P_n^t(dx|\gamma) \biggr| \nonumber \\
&\leq (1-\beta) \biggl( \limsup_{n\rightarrow\infty} \sum_{t=0}^{T} \beta^t  \biggl| \int_{\sX} g_{n,\overline{\varphi}_n}(x) R_n^t(dx|\gamma) \nonumber \\
&\phantom{xxxxxxxx}- \int_{\sX} g_{\overline{\varphi}_n}(x) P_n^t(dx|\gamma) \biggr| + 2 \sum_{t=T+1}^{\infty} \beta^t \|g\| \biggr) \nonumber
\end{align}
Since the first term in the last expression converges to zero as $n\rightarrow\infty$ for any $T$ by Lemma~\ref{lemma3} and the second term in the last expression converges to zero as $T\rightarrow\infty$ by $\|g\|<\infty$, the proof is complete.
\end{proof}

The below theorem is the main result of this section.

\begin{theorem}\label{main2}
For any given $\kappa>0$, there exist $\varepsilon>0$ and $n\geq1$ such that if
$\overline{\varphi}_n$ is an optimal policy for $(\overline{\textbf{CP}}_n^{\bk - \varepsilon \mathbf{1}})$ obtained by extending an optimal policy $\varphi_n$ for $($\textbf{CP}$_n^{\bk - \varepsilon \mathbf{1}})$ to $\sX$, then $\overline{\varphi}_n \in \Phi$ is feasible for $(\textbf{CP}^{\bk})$ and the true cost of $\overline{\varphi}_n$ is within $\kappa$ of the optimal value of $(\textbf{CP}^{\bk})$.
\end{theorem}

\begin{proof}
We observe that one can recover all the results derived in Section~\ref{app_val_disc} if the constraint vector $\bk$ is replaced by $\bk - \varepsilon \mathbf{1}$, where $\varepsilon>0$ satisfies
\begin{align}
\varepsilon < \min_{l=1,\ldots,q} \alpha_l'.  \label{neweq11}
\end{align}
Here, $\balpha'$ is the vector in Assumption~\ref{as0}-(c). Similar to the set $\C_n$ and the function $\V_n$ in the proof of Proposition~\ref{prop1}, define the set $\C \subset \R^q$ and the function $\V: \C \rightarrow \R$ as
\begin{align}
\C &= \bigcup_{\zeta \in \P} \biggl \{ \bm \in \R^q: \langle \zeta, c \rangle + \balpha = \bm \text{  for some  } \balpha \in \R^q_+ \biggr\}, \nonumber \\
\V(\bm) &= \min \biggl\{ \langle \zeta,c \rangle: \zeta \in \P \text{  and  } \langle \zeta, \bd \rangle + \balpha = \bm, \text{  } \balpha \in \R^q_+ \biggr\}. \nonumber
\end{align}
Hence, $\C$ is a convex subset of $\R^q$ and $\V$ is a convex function. We also have  $\V(\bk-\varepsilon\mathbf{1}) = \min(\textbf{CP}^{\bk - \varepsilon \mathbf{1}})$ for any $\varepsilon \geq 0$. Since $\bk \in \intr \C$, the function $\V$, being convex, is continuous at $\bk$. This implies the existence of a sequence $\{\varepsilon_k\}_{k\geq1}$ of positive real numbers such that (i) $\varepsilon_k$ satisfies (\ref{neweq11}) for all $k$, (ii) $\lim_{k\rightarrow\infty} \varepsilon_k = 0$, and therefore, $\lim_{k\rightarrow\infty} \V(\bk - \varepsilon_k \mathbf{1}) = \V(\bk)$.

Given any $\kappa > 0$, we choose $k\geq1$ sufficiently large such that
\begin{align}
\bigl| \V(\bk) - \V(\bk - \varepsilon_k \mathbf{1}) \bigr| < \frac{\kappa}{3}. \label{neweq14}
\end{align}
Then, for $\bk - \varepsilon_k \mathbf{1}$, we choose $n$ sufficiently large such that
\begin{align}
&\bigl| \V(\bk - \varepsilon_k \mathbf{1}) - \min(\textbf{CP}_n^{{\bk - \varepsilon_k \mathbf{1}}}) \bigl| < \frac{\kappa}{3} \label{neweq15} \\
&\bigl| J^c(\overline{\varphi}_n,\gamma) - J^c_n(\overline{\varphi}_n,\gamma) \bigl| = \bigl| J(\overline{\varphi}_n,\gamma) - \min(\textbf{CP}_n^{{\bk - \varepsilon_k \mathbf{1}}}) \bigl| < \frac{\kappa}{3} \label{neweq16}  \\
&\bigl| J^{d_l}(\overline{\varphi}_n,\gamma) - J^{d_l}_n(\overline{\varphi}_n,\gamma) \bigl| < \varepsilon_k \text{   } \text{for} \text{   } l=1,\ldots,q, \label{neweq17}
\end{align}
where $\overline{\varphi}_n$ is the optimal policy for ($\overline{\textbf{CP}}_n^{{\bk - \varepsilon_k \mathbf{1}}}$) obtained by extending the optimal policy $\varphi_n$ of (\textbf{CP}$_n^{{\bk - \varepsilon_k \mathbf{1}}}$) to $\sX$, i.e., $\overline{\varphi}_n(\,\cdot\,|x) = \varphi(\,\cdot\,|Q_n(x))$.
Here, (\ref{neweq15}) follows from Theorem~\ref{main1}; (\ref{neweq16}) and (\ref{neweq17}) follow from Proposition~\ref{prop2}. We observe that (\ref{neweq17}) implies that $\overline{\varphi}_n$ is feasible for (\textbf{CP}$^{\bk}$), and furthermore, by (\ref{neweq14}), (\ref{neweq15}), and (\ref{neweq16}), the true cost of $\overline{\varphi}_n$ is within $\kappa$ of the optimal value of (\textbf{CP}$^{\bk}$), i.e.,
\begin{align}
\bigl| J(\overline{\varphi}_n,\gamma) - \min(\textbf{CP}) \bigr| < \kappa. \nonumber
\end{align}
\end{proof}

\section{Asymptotic Approximation of Average Cost Problems}\label{constrained:ave}

In this section we obtain asymptotic approximation results, analogous to Theorems~\ref{main1} and \ref{main2}, for the average cost criterion. To achieve this, we impose the following assumptions.

\begin{assumption}
\label{as3}
Suppose Assumption~\ref{as0} holds and the stochastic kernel $p(\,\cdot\,|x,a)$ is continuous in $(x,a)$ with respect to the total variation distance. In addition, suppose there exist $\lambda \in \P(\sX)$, $\alpha \in (0,1)$, and $\phi \in B(\sX \times \sA)$ such that
\begin{itemize}
\item [(a)] $p(D|x,a) \geq \lambda(D) \phi(x,a)$ for all $D \in \B(\sX)$,
\item [(b)] $1-\alpha \leq  \phi(x,a)$.
\end{itemize}
\end{assumption}

Note that if we define $w \equiv 1$, then condition (b) corresponds to the so-called `drift inequality': for all $(x,a) \in \sX \times \sA$
\begin{align}
\int_{\sX} w(y) p(dy|x,a) &\leq \alpha w(x) + \int_{\sX} w(x) \lambda(dx) \phi(x,a), \nonumber
\end{align}
and condition (a) corresponds to the so-called `minorization' condition, both of which were used in literature for studying geometric ergodicity of Markov chains (see \cite{HeLa99,MeTw93}, and references therein).

\begin{example}\label{exm3}
Recall the model in Example~\ref{exm1}. We assume that the conditions in Example~\ref{exm2} hold. Verification of Assumption~\ref{as3}-(a),(b) is highly dependent on the systems components, and so, it is quite difficult to find a global assumption in order to satisfy Assumption~\ref{as3}-(a),(b). One way to establish this is as follows. Suppose that $F(x,a,v)$ has the following form: $F(x,a,v) = H(x,a) + v$. This is called an `additive-noise model'. In this case, Assumption~\ref{as3}-(a),(b) is true if the density of the noise is strictly positive.
\end{example}

Recall that any randomized stationary policy $\varphi$ defines a stochastic kernel
\begin{align}
p(\,\cdot\,|x,\varphi) \coloneqq \int_{\sA} p(\,\cdot\,|x,a) \varphi(da|x) \nonumber
\end{align}
on $\sX$ given $\sX$. For any $t\geq1$, we let $p^t(\,\cdot\,|x,\varphi)$ denote the $t$-step transition probability of this Markov chain given the initial point $x$.


The below theorem is a consequence of \cite[Theorem 3.3]{Veg03}, \cite[Lemma 3.4]{GoHe95}, and \cite[Theorem 3]{JaNo06}. In what follows, for any $g \in B(\sX \times \sA)$ and $\varphi \in \Phi$, we let
$g_{\varphi}(x) \coloneqq \int_{\sA} g(x,a) \varphi(da|x)$.

\begin{theorem}\label{compact:thm4}
For any $\varphi \in \Phi$, the stochastic kernel $p(\,\cdot\,|x,\varphi)$ has a unique invariant probability measure $\mu_{\varphi}$
and we have
\begin{align}
V(\varphi,\gamma) = \langle \mu_{\varphi}, c_{\varphi} \rangle
\text{ and }
V_l(\varphi,\gamma) = \langle \mu_{\varphi}, d_{l,\varphi} \rangle, \text{  } l=1,\ldots,q. \nonumber
\end{align}
Furthermore, there exist positive real numbers $R$ and $\kappa < 1$ such that for every $x \in \sX$
\begin{align}
\sup_{\varphi \in \Phi} \| p^t(\,\cdot\,|x,\varphi) - \mu_{\varphi} \|_{TV} \leq  R \kappa^t. \nonumber
\end{align}
Finally, for any one-stage cost function $g$, there exists $h^* \in B(\sX)$ such that the average cost optimality equality (ACOE) holds:
\begin{align}
\rho^{*} + h^{*}(x) &= \min_{a\in\sA} \biggl[ g(x,a) + \int_{\sX} h^{*}(y) p(dy|x,a) \biggr], \nonumber
\end{align}
where
$
\rho^* = \inf_{\varphi \in \Phi} V^g(\varphi,\gamma).
$
\end{theorem}

Theorem~\ref{compact:thm4} and \cite[Lemma 5.7.10]{HeLa96} imply that (\textbf{CP}$^{\bk}$) is equivalent to the following optimization problem, which is also denoted by (\textbf{CP}$^{\bk}$):
\begin{align}
(\textbf{CP}^{\bk}) \text{                         }&\text{minimize}_{\varphi \in \Phi} \langle \mu_{\varphi}, c_{\varphi} \rangle
\nonumber \\*
&\text{subject to  } \langle \mu_{\varphi}, d_{l,\varphi} \rangle \leq k_l \text{ } (l= 1, \ldots, q).
\nonumber
\end{align}
Furthermore, we define
\begin{align}
\phi_n(x,a) &\coloneqq \int \phi(y,a) \nu_{n,i_n(x)}(dy), \nonumber \\
\lambda_n &\coloneqq Q_n \ast \lambda. \nonumber
\end{align}
Then $\overline{\text{MDP}}_n$ satisfies Assumption~\ref{as3}-(a),(b) when $\phi$ is replaced by $\phi_n$, and Assumption~\ref{as3}-(a),(b) is true for MDP$_n$ when $\phi$ and $\lambda$ are replaced by the restriction of $\phi_n$ to $\sX_n$ and $\lambda_n$, respectively. Hence, Theorem~\ref{compact:thm4} holds (with the same $R$ and $\kappa$) for $\overline{\text{MDP}}_n$ and MDP$_n$ for all $n$. We denote by $\bar{\mu}^n_{\varphi}$ and $\mu^n_{\varphi}$ the invariant probability measures of $\overline{\text{MDP}}_n$ and MDP$_n$ corresponding to the policy $\varphi$, respectively. Therefore, the average cost constrained problems for MDP$_n$ and $\overline{\text{MDP}}_n$ are equivalent to the following optimization problems, respectively:
\begin{align}
\textbf{(CP$_n^{\bk}$)} \text{                         }&\text{minimize}_{\varphi \in \Phi_n} \langle \mu^n_{\varphi}, c_{n,\varphi} \rangle
\nonumber \\*
&\text{subject to  } \langle \mu^n_{\varphi}, d_{l,n,\varphi} \rangle \leq k_l \text{ } (l= 1, \ldots, q). \nonumber
\end{align}
\begin{align}
(\overline{\textbf{CP}}_n^{\bk}) \text{                         }&\text{minimize}_{\varphi \in \Phi} \langle \bar{\mu}^n_{\varphi}, b_{n,\varphi} \rangle
\nonumber \\*
&\text{subject to  } \langle \bar{\mu}^n_{\varphi}, r_{l,n,\varphi} \rangle \leq k_l \text{ } (l= 1, \ldots, q). \nonumber
\end{align}

The following lemma can be proved similar to Lemma~\ref{lemma3}. Hence, we omit the proof.

\begin{lemma}\label{lemma5}
For all $t\geq1$, we have
\begin{align}
\lim_{n\rightarrow\infty} \sup_{(y,\varphi) \in \sX\times\Phi} \bigl\|p^t(\,\cdot\,|y,\varphi) - q_n^t(\,\cdot\,|y,\varphi) \bigr\|_{TV} = 0. \nonumber
\end{align}
\end{lemma}

Using Lemma~\ref{lemma5}, one can prove the following result.

\begin{lemma}\label{ave:asym}
We have
\begin{align}
\lim_{n\rightarrow\infty} \sup_{\varphi \in \Phi} \bigl| \langle \mu_{\varphi}, c_{\varphi} \rangle - \langle \bar{\mu}^n_{\varphi}, b_{n,\varphi} \rangle \bigr| &= 0, \label{cons:con} \\
\lim_{n\rightarrow\infty} \sup_{\varphi \in \Phi} \bigl| \langle \mu_{\varphi}, d_{l,\varphi} \rangle - \langle \bar{\mu}^n_{\varphi}, r_{l,n,\varphi} \rangle \bigr| &= 0, \text{  } l=1,\ldots,q. \label{cons:con2}
\end{align}
\end{lemma}

\begin{proof}
We only prove (\ref{cons:con}) since the proof of (\ref{cons:con2}) is identical. Note that we have
\begin{align}
&\sup_{\varphi \in \Phi} \bigl| \langle \mu_{\varphi}, c_{\varphi} \rangle - \langle \bar{\mu}^n_{\varphi}, b_{n,\varphi} \rangle \bigr| \nonumber \\
&\leq \sup_{\varphi \in \Phi} \bigl| \langle \mu_{\varphi}, c_{\varphi} \rangle - \langle \bar{\mu}^n_{\varphi}, c_{\varphi} \rangle \bigr| + \sup_{\varphi \in \Phi} \bigl| \langle \bar{\mu}^n_{\varphi}, c_{\varphi} \rangle - \langle \bar{\mu}^n_{\varphi}, b_{n,\varphi} \rangle \bigr|. \nonumber
\end{align}
It is straightforward to show that $b_{n,\varphi} \rightarrow c_{\varphi}$ uniformly. Hence, the second term in the right side of the above equation goes to zero as $n\rightarrow\infty$. For the first term, we have, for any $t\geq1$ and $y \in \sX$,
\begin{align}
&\sup_{\varphi \in \Phi} \bigl| \langle \mu_{\varphi}, c_{\varphi} \rangle - \langle \bar{\mu}^n_{\varphi}, c_{\varphi} \rangle \bigr| \nonumber \\
&\leq \sup_{\varphi \in \Phi} \biggl | \int_{\sX} c_{\varphi}(x)  \mu_{\varphi}(dx) - \int_{\sX} c_{\varphi}(x) p^t(dx|y,\varphi) \biggr | \nonumber \\
&\phantom{xxxx}+ \sup_{\varphi \in \Phi} \biggl | \int_{\sX} c_{\varphi}(x) p^t(dx|y,\varphi) - \int_{\sX} c_{\varphi}(x) q_n^t(dx|y,\varphi) \biggr | \nonumber \\
&\phantom{xxxxxxxxx}+ \sup_{\varphi \in \Phi} \biggl | \int_{\sX} c_{\varphi}(x) q_n^t(dx|y,\varphi) - \int_{\sX} c_{\varphi}(x) \bar{\mu}^n_{\varphi}(dx) \biggr | \nonumber \\
&\leq 2 R \kappa^t \|c\|  + \|c\| \sup_{(y,\varphi)\in\sX\times\Phi} \bigl\| q_n^t(\,\cdot\,|y,\varphi) - p^t(\,\cdot\,|y,\varphi) \bigr\|_{TV}, \nonumber
\end{align}
where $R$ and $\kappa$ are the constants in Theorem~\ref{compact:thm4}. Then, the result follows from Lemma~\ref{lemma5}.
\end{proof}

Therefore, (\ref{cons:con2}) and Assumption~\ref{as0}-(c) imply that there exists $n_f \in \mathbb{N}$ such that for $n \geq n_f$, the problem $(\overline{\textbf{CP}}_n^{\bk})$ is consistent; that is, there exists a policy $\varphi \in \Phi$ which \emph{strictly} satisfies the constraints in $(\overline{\textbf{CP}}_n^{\bk})$. By \cite[Theorem 4.5]{HeGoLo03}, we can also conclude that $($\textbf{CP}$_n^{\bk})$ is also consistent. Then, by \cite[Theorem 3.2]{HeGoLo03}, we have the following result.

\begin{theorem}
Both $(\textbf{CP}^{\bk})$, $($\textbf{CP}$_n^{\bk})$, and $(\overline{\textbf{CP}}_n^{\bk})$ are solvable; that is, there exist optimal policies for each problem.
\end{theorem}

In the remainder of this section, it is assumed that $n \geq n_f$. Analogous to Lemma~\ref{lemma1}, the following result states that MDP$_n$ and $\overline{\text{MDP}}_n$ are essentially equivalent for the average cost.

\begin{lemma}\label{lemma4}
We have
\begin{align}
\min (\textbf{CP}_n^{\bk}) = \min (\overline{\textbf{CP}}_n^{\bk}), \nonumber
\end{align}
and if the randomized stationary policy $\varphi^* \in \Phi_n$ is optimal for $($\textbf{CP}$_n^{\bk})$, then its extension $\overline{\varphi}^*$ to $\sX$ is also optimal for $(\overline{\textbf{CP}}_n^{\bk})$ with the same cost function.
\end{lemma}

The proof of Lemma~\ref{lemma4} is given in Appendix~\ref{prooflemma4}. By Lemma~\ref{lemma4}, in the remainder of this section we consider $\overline{\text{MDP}}_n$ in place of MDP$_n$. For any $\bm \in \R^q$, we define
\begin{align}
\Delta(\bm) &\coloneqq \bigl\{ \varphi \in \Phi: \langle \mu_{\varphi}, d_{l,\varphi} \rangle \leq m_l, \text{  } l=1,\ldots,q \bigr\}  \nonumber \\
\Delta_n(\bm) &\coloneqq \bigl\{ \varphi \in \Phi: \langle \bar{\mu}^n_{\varphi}, r_{l,n,\varphi} \rangle \leq m_l, \text{  } l=1,\ldots,q \bigr\}. \nonumber
\end{align}
Then, we let $\C \coloneqq \bigl\{ \bm \in \R^q: \Delta(\bm) \neq \emptyset \bigr\}$ and $\C_n \coloneqq \bigl\{ \bm \in \R^q: \Delta_n(\bm) \neq \emptyset\bigr\}$. It can be proved that both $\C$ and $\C_n$ are convex subsets of $\R^q$. Let us also define functions $\W$ and $\W_n$ over $\C$ and $\C_n$, respectively, as follows:
\begin{align}
\W(\bm) &\coloneqq \min \bigl\{ \langle \mu_{\varphi}, c_{\varphi} \rangle : \varphi \in \Delta(\bm) \bigr\} \nonumber \\
\W_n(\bm) &\coloneqq \min \bigl\{ \langle \bar{\mu}^n_{\varphi}, b_{n,\varphi} \rangle : \varphi \in \Delta_n(\bm) \bigr\}. \nonumber
\end{align}
It can also be proved that both $\W$ and $\W_n$ are convex functions. Note that $\min(\textbf{CP}^{\bk}) = \W(\bk)$ and $\min(\overline{\textbf{CP}}_n^{\bk}) = \W_n(\bk)$. Furthermore, by Assumption~\ref{as0}-(c) we have $\bk \in \sint \text{ } \C$. Since, $n\geq n_f$, we also have $\bk \in \sint \text{ } \C_n$. Therefore, functions $\W_n$ and $\W$, being convex, are continuous at $\bk$.

The following theorem is analogous to Theorem~\ref{main1} and states that the optimal value of $(\textbf{CP}_n^{\bk})$ (or equivalently, the optimal value of ($\overline{\textbf{CP}}_n^{\bk}$)) converges to the optimal value of $(\textbf{CP}^{\bk})$.

\begin{theorem}\label{cons:average:main1}
We have
\begin{align}
\lim_{n\rightarrow\infty} \bigl| \min(\textbf{CP}_n^{\bk}) - \min(\textbf{CP}^{\bk}) \bigr|. \nonumber
\end{align}
\end{theorem}

\begin{proof}
The result follows from (\ref{cons:con}) and (\ref{cons:con2}), and the fact that $\W_n$ and $\W$ are continuous at $\bk$.
\end{proof}

The following theorem is analogous to Theorem~\ref{main2} and is the main result of this section. It establishes a method for computing near optimal policies for the constrained average-cost Markov decision problem (\textbf{CP}$^{\bk}$).

\begin{theorem}\label{cons:average:main2}
For any given $\kappa>0$, there exist $\varepsilon>0$ and $n\geq1$ such that if
$\overline{\varphi}_n$ is an optimal policy for $(\overline{\textbf{CP}}_n^{\bk - \varepsilon \mathbf{1}
})$ obtained by extending an optimal policy $\varphi_n$ for $($\textbf{CP}$_n^{\bk - \varepsilon \mathbf{1}
})$ to $\sX$, then $\overline{\varphi}_n \in \Phi$ is feasible for $($\textbf{CP}$^{\bk})$ and the true cost of $\overline{\varphi}_n$ is within $\kappa$ of the optimal value of $($\textbf{CP}$^{\bk})$.
\end{theorem}

\begin{proof}
The result follows from (\ref{cons:con}) and (\ref{cons:con2}), the fact that $\W_n$ and $\W$ are continuous at $\bk$, and Theorem~\ref{cons:average:main1}. It can be done similar to the proof of Theorem~\ref{main2}, and so, we omit the details.
\end{proof}

\section{Rate of Convergence Analysis}
\label{constrained:rateconv}

In this section we derive upper bounds on the performance losses due to discretization in terms of the cardinality of the set $\sX_n$ (i.e., number of grid points). To do this, we will impose the following additional assumptions on the components of the MDP for both discounted cost and average cost criteria.

\begin{assumption}
\label{as4}
We suppose that Assumption~\ref{as0} holds. Furthermore, we assume the following.
\begin{itemize}
\item[(a)] The one-stage cost function $c$ and the constraint functions $d_l$ ($l=1,\ldots,q$) satisfy $c(\,\cdot\,,a) \in \Lip(\sX,K_c)$ and $d_l(\,\cdot\,,a) \in \Lip(\sX,K_l)$ for all $a \in \sA$ for some $K_c$ and $K_l$.
\item[(b)] $\sX$ is a compact subset of $\R^d$ for some $d\geq1$, equipped with the Euclidean norm.
\end{itemize}
\end{assumption}

We note that Assumption~\ref{as4}-(b) implies the existence of a constant $\alpha>0$ and finite subsets $\sX_n\subset\sX$ with cardinality $n$ such that
\begin{align}
\max_{x\in\sX}\min_{y\in \sX_n} d_{\sX}(x,y)\leq \alpha (1/n)^{1/d}  \label{quantcof}
\end{align}
for all $n$, where $d_{\sX}$ is the Euclidean distance on $\sX$. In the remainder of this section, we replace $\sX_n$ defined in Section~\ref{sec3} with $\sX_n$ satisfying (\ref{quantcof}) in order to derive \emph{explicit} bounds on the approximation error in terms of the cardinality of $\sX_n$.

Recall that, in this paper, finite models are obtained through quantizing the state space \cite{Whi78,Whi79} instead of randomly sampling it as in \cite{DuPr14-a,MuSz08}.

\subsection{Discounted Cost: Approximation of optimal value}
\label{discountval:rateconv}

In this section, we establish an upper bound on the error of the approximation of the discounted optimal value. Let $W_1$ denote the \emph{Wasserstein distance of order $1$} \cite[p. 95]{Vil09}. Note that, for compact $\sX$, $W_1$ metrizes the weak topology on $\P(\sX)$ \cite[Corollary 6.13, p. 97]{Vil09}. The following assumptions will be imposed in addition to Assumption~\ref{as4}.

\begin{assumption}\label{as5}
Assumption~\ref{as1}-(a) holds. Furthermore, we assume
\begin{itemize}
\item[(a)] The stochastic kernel $p$ satisfies
\begin{align}
W_1\bigl(p(\,\cdot\,|x,a),p(\,\cdot\,|y,a)\bigr) \leq K_p d_{\sX}(x,y), \nonumber
\end{align}
for all $a \in \sA$ for some $K_p$.
\item[(b)] $K_p \beta < 1$.
\end{itemize}
\end{assumption}

\begin{example}\label{exm4}
Recall the model in Example~\ref{exm1}. We assume that $F(x,a,v)$, one-stage cost function $c(x,a)$, and constraint functions $d_l(x,a)$ are uniformly Lipschitz in $x$ for all $a$ and $v$. In addition, Lipschitz constant $K_{F}$ of $F$ satisfies $K_{F} \beta < 1$. Then, Assumptions~\ref{as4} and \ref{as5} hold.
\end{example}

The following theorem is the main result of this section.

\begin{theorem}\label{main3}
We have
\begin{align}
\bigl|\min (\textbf{CP$_n^{\bk}$}) - \min (\textbf{CP}^{\bk})\bigr| \leq Y_v (1/n)^{1/d} \nonumber
\end{align}
where
\begin{align}
Y_v = \frac{4\alpha(K_c+ q K K_l)}{1-\beta K_p} \nonumber
\end{align}
and $K$ is the constant in Proposition~\ref{prop1}.
\end{theorem}

\begin{proof}
Recall the definitions of functions $G_n: \R^q_- \rightarrow \R$ and $G: \R^q_- \rightarrow \R$ given by
\begin{align}
G_n(\bdelta) &= (1-\beta)\langle \gamma,\overline{u}^*_{n,\bdelta} \rangle + \langle \bk,\bdelta \rangle \nonumber \\
G(\bdelta) &= (1-\beta)\langle \gamma,u^*_{\bdelta} \rangle + \langle \bk,\bdelta \rangle. \nonumber
\end{align}
Recall also that
\begin{align}
\max(\overline{\textbf{CP}}_n^{*,\bk}) = \sup_{\bdelta \in \K} G_n(\bdelta) \text{ and } \max(\textbf{CP}^{*,\bk}) = \sup_{\bdelta \in \K} G(\bdelta), \nonumber
\end{align}
where $\K = \{\bdelta \in \R^q_-: \| \bdelta \|_1 \leq K\}$. By \cite[Theorem 5.2]{SaYuLi17}, we have
\begin{align}
\| \overline{u}^*_{n,\bdelta} - u^*_{\bdelta} \| \leq \frac{1}{1-\beta} \frac{K_c+\sum_{l=1}^q \delta_l K_l}{1-\beta K_p} 4\alpha (1/n)^{1/d} \nonumber
\end{align}
for all $\bdelta \in \K$. This implies that
\begin{align}
|G_n(\bdelta) - G(\bdelta)| \leq \frac{K_c+\sum_{l=1}^q \delta_l K_l}{1-\beta K_p} 4\alpha (1/n)^{1/d}. \nonumber
\end{align}
Then, we have
\begin{align}
\bigl | \min (\overline{\textbf{CP}}_n^{\bk}) - \min (\overline{\textbf{CP}}^{\bk}) \bigr| &= \bigl | \max (\overline{\textbf{CP}}_n^{*,\bk}) - \max (\overline{\textbf{CP}}^{*,\bk}) \bigr| \nonumber \\
&=\bigl| \sup_{\bdelta \in \K} G_n(\bdelta) - \sup_{\bdelta \in \K} G(\bdelta) \bigr| \nonumber \\
&\leq \sup_{\bdelta \in \K} \bigl| G_n(\bdelta) - G(\bdelta) \bigr| \nonumber \\
&\leq \sup_{\bdelta \in \K} \frac{K_c+\sum_{l=1}^q \delta_l K_l}{1-\beta K_p} 4\alpha (1/n)^{1/d} \nonumber \\
&\leq \frac{K_c+ q K K_l}{1-\beta K_p} 4\alpha (1/n)^{1/d}. \nonumber
\end{align}
\end{proof}

\begin{remark}
It is important to point out that if we replace Assumption~\ref{as5}-(a) with the uniform Lipschitz continuity of $p(\,\cdot\,|x,a)$ in $x$ with respect to total variation distance, then Theorem~\ref{main3} remains valid (with possibly different constant $H_v$ in front of the term $(1/n)^{1/d}$). However, in this case, we do not need the assumption $K_p \beta < 1$.
\end{remark}

\subsection{Discounted Cost: Approximation of Optimal Policy}
\label{discountpol:rateconv}

In this section, an upper bound on the error of the approximation of the optimal policy for discounted cost will be derived. We impose the following conditions in addition to Assumption~\ref{as2}-(a) and Assumption~\ref{as4}.

\begin{assumption}\label{as6}
\begin{itemize}
\item [ ]
\item[(a)] The stochastic kernel $p$ satisfies
\begin{align}
\| p(\,\cdot\,|x,a)- p(\,\cdot\,|y,a) \|_{TV} \leq G_{p} d_{\sX}(x,y), \nonumber
\end{align}
for all $a \in \sA$ for some $G_{p}$.
\end{itemize}
\end{assumption}

\begin{example}\label{exm5}
Recall the model in Example~\ref{exm1}. We assume that the transition probability $p(dy|x,a)$ has a density $f(y|x,a)$ which is uniformly Lipschitz continuous in $x$ for all $(y,a)$. Then, Assumption~\ref{as6} holds.
\end{example}

Using Assumption~\ref{as6}, we first derive an upper bound for the asymptotic convergence result in Proposition~\ref{prop2} when $g$ is Lipschitz continuous. Recall the definitions we have made in Section~\ref{app_pol_disc}.

\begin{proposition}\label{upper:prop2}
Let $\{\varphi_n\}$ be a sequence such that $\varphi_n \in \Phi_n$ for all $n$. Then, for any $g \in \Lip(\sX,K_g)$, we have
\begin{align}
&\bigl| J^{g_n}_n(\overline{\varphi}_n,\gamma) - J^g(\overline{\varphi}_n,\gamma) \bigr| \leq H_g (1/n)^{1/d}, \nonumber \\
\intertext{where}
&H_g = \biggl( K_g + \frac{\|g\| G_p}{1-\beta} \biggr) 2 \alpha (1/n)^{1/d}. \nonumber
\end{align}
\end{proposition}

The proof of Proposition~\ref{upper:prop2} is given in Appendix~\ref{proofupper:prop2}. The below theorem is the main result of this section.

\begin{theorem}\label{main4}
Given any $\kappa>0$, let $\varepsilon = \frac{\kappa}{3K} < \min_{l=1,\ldots,q} \alpha_l'$ and
\begin{align}
n \geq \max\biggl( \biggl( \frac{3 Y_v}{\kappa} \biggr)^d, \biggl( \frac{3 H_c}{\kappa} \biggr)^d, \biggl( \frac{3 H_l^{^{\max}} K}{\kappa} \biggr)^d \biggr), \label{min:grid}
\end{align}
where $H_l^{^{\max}} \coloneqq \max_{l=1,\ldots,q} H_{d_l}$. If $\overline{\varphi}_n$ is an optimal policy for $(\overline{\textbf{CP}}_n^{\bk - \varepsilon \mathbf{1}})$ obtained by extending an optimal policy $\varphi_n$ for $($\textbf{CP}$_n^{\bk - \varepsilon \mathbf{1}})$ to $\sX$, then $\overline{\varphi}_n \in \Phi$ is feasible for $($\textbf{CP}$^{\bk})$ and the true cost of $\overline{\varphi}_n$ is within $\kappa$ of the optimal value of $($\textbf{CP}$^{\bk})$.
\end{theorem}

\begin{proof}
For any $\varepsilon < \min_{l=1,\ldots,q} \alpha_l'$, we have
\begin{align}
\bigl | \min (\textbf{CP}^{\bk}) - \min (\textbf{CP}^{\bk - \varepsilon \mathbf{1}}) \bigr| &\leq \sup_{\bdelta \in \K} \bigl| (1-\beta)\langle \gamma,u^*_{\bdelta} \rangle + \langle \bk,\bdelta \rangle \nonumber \\
&- (1-\beta)\langle \gamma,u^*_{\bdelta} \rangle - \langle \bk-\varepsilon\mathbf{1},\bdelta \rangle \bigr| \nonumber \\
&\leq K \varepsilon. \nonumber
\end{align}
Hence, we have
\begin{align}
\bigl| \min (\textbf{CP}^{\bk}) - \min (\textbf{CP}^{\bk - \varepsilon \mathbf{1}}) \bigr| < \frac{\kappa}{3}. \label{rate:neweq14}
\end{align}
Furthermore, by (\ref{min:grid}), Theorem~\ref{main3}, and Proposition~\ref{upper:prop2}, we also have
\begin{align}
&\bigl| \min (\textbf{CP}^{\bk - \varepsilon \mathbf{1}}) - \min (\overline{\textbf{CP}}_n^{\bk - \varepsilon \mathbf{1}}) \bigl| < \frac{\kappa}{3} \label{rate:neweq15} \\
&\bigl| J^c(\overline{\varphi}_n,\gamma) - J^c_n(\overline{\varphi}_n,\gamma) \bigl| = \bigl| J(\overline{\varphi}_n,\gamma) - \min(\textbf{CP}_n^{\bk - \varepsilon \mathbf{1}}) \bigl| < \frac{\kappa}{3} \label{rate:neweq16}  \\
&\bigl| J^{d_l}(\overline{\varphi}_n,\gamma) - J^{d_l}_n(\overline{\varphi}_n,\gamma) \bigl| < \varepsilon \text{   } \text{for} \text{   } l=1,\ldots,q, \label{rate:neweq17}
\end{align}
where $\overline{\varphi}_n$ is the optimal policy for ($\overline{\textbf{CP}}_n^{\bk - \varepsilon \mathbf{1}}$) obtained by extending the optimal policy $\varphi_n$ of (\textbf{CP}$_n^{\bk - \varepsilon \mathbf{1}}$) to $\sX$, i.e., $\overline{\varphi}_n(\,\cdot\,|x) = \varphi(\,\cdot\,|Q_n(x))$. Here, (\ref{rate:neweq15}) follows from Theorem~\ref{main3}; (\ref{rate:neweq16}) and (\ref{rate:neweq17}) follow from Proposition~\ref{upper:prop2}.
In view of the proof of Theorem~\ref{main2}, this completes the proof.
\end{proof}

\subsection{Average Cost: Approximation of optimal value}
\label{averval:rateconv}

In this section, an upper bound on the error of the approximation of the optimal value for the average cost will be derived. Assumption~\ref{as3}, Assumption~\ref{as4}, and Assumption~\ref{as6} will be imposed in this section. To simplify the notation in the sequel, let us define the following constants:
\begin{align}
I_1 = 2 \|c\| R, \text{ } I_2 = 2 K_c \alpha, \text{ } I_3 = 2 \|c\| G_p \alpha, \text{ } I_4 = \frac{I_3}{I_1 \ln(\frac{1}{\kappa})}. \nonumber
\end{align}
Similarly, we define constants $I_1^l$, $I_2^l$, $I_3^l$, and $I_4^l$ by replacing $c$ with $d_l$ for each $l=1,\ldots,q$. Before stating the main theorem, we obtain the following upper bounds on the asymptotic convergence results in Lemma~\ref{ave:asym}. The proof of this result is given in Appendix~\ref{proofupper:cost}.

\begin{lemma}\label{upper:cost}
We have
\begin{align}
&\sup_{\varphi \in \Phi} \bigl| \langle \mu_{\varphi}, c_{\varphi} \rangle - \langle \bar{\mu}^n_{\varphi}, b_{n,\varphi} \rangle \bigr| \nonumber \\
&\phantom{xxxxx}\leq (I_1 I_4 + I_2) (1/n)^{1/d} + \frac{I_3}{\ln(1/\kappa)} (1/n)^{1/d} \ln\bigl(\frac{n^{1/d}}{I_4}\bigr), \nonumber  \\
\intertext{and, for any $l=1,\ldots,q$, we have}
&\sup_{\varphi \in \Phi} \bigl| \langle \mu_{\varphi}, d_{l,\varphi} \rangle - \langle \bar{\mu}^n_{\varphi}, r_{l,n,\varphi} \rangle \bigr| \nonumber \\
&\phantom{xxxxx}\leq(I_1^l I_4^l + I_2^l) (1/n)^{1/d} + \frac{I_3^l}{\ln(1/\kappa)} (1/n)^{1/d} \ln\bigl(\frac{n^{1/d}}{I_4^l}\bigr). \nonumber
\end{align}
\end{lemma}

For each $n$, to ease the notation, let us also define the following constants:
\begin{align}
\varepsilon_c(n) &\coloneqq (I_1 I_4 + I_2) (1/n)^{1/d} + \frac{I_3}{\ln(1/\kappa)} (1/n)^{1/d} \ln\bigl(\frac{n^{1/d}}{I_4}\bigr) \nonumber \\
\varepsilon_{d_l}(n) &\coloneqq (I_1^l I_4^l + I_2^l) (1/n)^{1/d} + \frac{I_3^l}{\ln(1/\kappa)} (1/n)^{1/d} \ln\bigl(\frac{n^{1/d}}{I_4^l}\bigr) \nonumber \\
\varepsilon_{_{\max}}(n) &\coloneqq \max_{l=1,\ldots,q} \varepsilon_{d_l}(n). \nonumber
\end{align}
The following theorem is the main result of this section.

\begin{theorem}\label{main5}
Suppose that $\varepsilon_{_{\max}}(n) < \frac{1}{2} \min_{l=1,\ldots,q} \alpha_l'$. Then, we have
\begin{align}
\bigl|\min (\textbf{CP$_n^{\bk}$}) - \min (\textbf{CP}^{\bk})\bigr| \leq 2 \varepsilon_c(n) + \frac{4\|c\|}{\alpha}\varepsilon_{_{\max}}(n). \nonumber
\end{align}
\end{theorem}

\begin{proof}

Note first that, by Lemma~\ref{upper:cost}, we have
\begin{align}
\min (\textbf{CP}^{\bk -\varepsilon_{_{\max}}(n) \mathbf{1}}) \geq \min (\textbf{CP}_n^{\bk}) - \varepsilon_c(n) \nonumber \\
\intertext{and}
\min (\textbf{CP}^{\bk +\varepsilon_{_{\max}}(n)\mathbf{1}} ) \leq \min (\textbf{CP}_n^{\bk}) + \varepsilon_c(n). \nonumber
\end{align}
We also have
\begin{align}
\min (\textbf{CP}^{\bk-\varepsilon_{_{\max}}(n)\mathbf{1}} ) \geq \min (\textbf{CP}^{\bk}) \geq \min (\textbf{CP}^{\bk +\varepsilon_{_{\max}}(n)\mathbf{1}} ). \nonumber
\end{align}
Therefore, we obtain
\begin{align}
\bigl|\min (\textbf{CP}_n^{\bk}) &- \min (\textbf{CP}^{\bk})\bigr| \nonumber \\
&\leq 2 \varepsilon_c(n) + \min (\textbf{CP}^{\bk-\varepsilon_{_{\max}}(n)\mathbf{1}} ) - (\textbf{CP}^{\bk +\varepsilon_{_{\max}}(n)\mathbf{1}} ). \nonumber
\end{align}
To bound the term $\min (\textbf{CP}^{\bk -\varepsilon_{_{\max}}(n)\mathbf{1}} ) - \min(\textbf{CP}^{\bk +\varepsilon_{_{\max}}(n)\mathbf{1}} )$, we use the dual problems.

Recall that, for any constraint vector $\bm$, the dual problem $\textbf{(CP$^{*,\bm}$)}$ is defined as \cite{HeGoLo03}:
\begin{align}
 \text{                         }&\text{maximize}_{(\bdelta,\delta_0,u) \in \R_{-}^q \times \R \times B(\sX)} \text{ }
\delta_0 + \bigl\langle \bm, \bdelta \bigr\rangle \nonumber
\end{align}
subject to
\begin{align}
u(x) + \delta_0 \leq c(x,a) - \sum_{l=1}^q \delta_l d_l(x,a) + \int_{\sX} u(y) p(dy|x,a) \nonumber
\end{align}
for all $(x,a) \in \sX\times\sA$. The following lemma is very similar to Proposition~\ref{prop1} and it will be proved in the Appendix.

\begin{lemma}\label{lemma6}
Suppose $\bm \geq \bk - \frac{\alpha}{2} \mathbf{1}$ where $\alpha = \min_{l=1,\ldots,q} \alpha_l'$. Then there exists a maximizer $(\delta_0^*,\bdelta^*) \in \R \times \R_{-}^q$ for \textbf{(CP$^{*,\bm}$)} such that $\|\bdelta^*\|_1 \leq \frac{2\|c\|}{\alpha}$.
\end{lemma}

By Lemma~\ref{lemma6}, if $\bm \geq \bk - \frac{\alpha}{2} \mathbf{1}$, we can write the dual problem $\textbf{(CP$^{*,\bm}$)}$ as follows:
\begin{align}
\text{                         }&\text{maximize}_{(\bdelta,\delta_0,u) \in \S \times \R \times B(\sX)}
\delta_0 + \bigl\langle \bm, \bdelta \bigr\rangle \nonumber
\end{align}
subject to
\begin{align}
u(x) + \delta_0 \leq c(x,a) - \sum_{l=1}^q \delta_l d_l(x,a) + \int_{\sX} u(y) p(dy|x,a) \nonumber
\end{align}
for all $(x,a) \in \sX\times\sA$, where
$\S \coloneqq \{\bdelta \in \R_{-}^q: \|\bdelta^*\|_1 \leq \frac{2\|c\|}{\alpha}\}$. Therefore, since there is no duality gap and the set of feasible points for \textbf{(CP$^{*,\bm}$)} does not depend on the constraint vector $\bm$, we have
\begin{align}
&\min (\textbf{CP}^{\bk-\varepsilon_{_{\max}}(n)\mathbf{1}} ) - \min(\textbf{CP}^{\bk+\varepsilon_{_{\max}}(n)\mathbf{1}} ) \nonumber \\
&\phantom{xxxxx}\leq \sup_{\bdelta \in \S} \bigl| \langle \bk -\varepsilon_{_{\max}}(n) \mathbf{1} , \bdelta \rangle - \langle \bk + \varepsilon_{_{\max}}(n) \mathbf{1} , \bdelta \rangle \bigr| \nonumber \\
&\phantom{xxxxx}\leq \sup_{\bdelta \in \S} 2 \langle \varepsilon_{_{\max}}(n) \mathbf{1} , \bdelta \rangle \nonumber \\
&\phantom{xxxxx}= \frac{4\|c\|}{\alpha} \varepsilon_{_{\max}}(n). \nonumber
\end{align}
\end{proof}

\subsection{Average Cost: Approximation of Optimal Policy}
\label{averpol:rateconv}

In this section, an upper bound on the error of the approximation of the optimal policy for the average cost will be derived. Assumption~\ref{as3}, Assumption~\ref{as4}, and Assumption~\ref{as6} will be imposed in this section. The below theorem is the main result of this section.

\begin{theorem}\label{main6}
Given any $\kappa>0$, let $\varepsilon = \frac{\kappa}{3 \tilde{K}} < \frac{1}{2} \min_{l=1,\ldots,q} \alpha_l'$ and let $n$ satisfy
\begin{align}
\frac{\kappa}{3} \geq 2 \varepsilon_c(n) + \max\{2 K,1\} \varepsilon_{_{\max}}(n), \label{min:ave:grid}
\end{align}
where $K = \frac{2\|c\|}{\alpha}$. If $\overline{\varphi}_n$ is an optimal policy for $(\overline{\textbf{CP}}_n^{\bk-\varepsilon\mathbf{1}} )$ obtained by extending an optimal policy $\varphi_n$ for $($\textbf{CP}$_n^{\bk-\varepsilon\mathbf{1}} )$ to $\sX$, then $\overline{\varphi}_n \in \Phi$ is feasible for $($\textbf{CP}$^{\bk})$ and the true cost of $\overline{\varphi}_n$ is within $\kappa$ of the optimal value of $($\textbf{CP}$^{\bk})$.
\end{theorem}

\begin{proof}
For any $\varepsilon < \frac{1}{2} \min_{l=1,\ldots,q} \alpha_l'$, we have
\begin{align}
\bigl | \min (\textbf{CP}^{\bk}) - \min (\textbf{CP}^{\bk-\varepsilon \mathbf{1}}) \bigr|
&\leq \sup_{\bdelta \in \S} \bigl| \langle \bk, \bdelta \rangle - \langle \bk - \varepsilon \mathbf{1} , \bdelta \rangle \bigr| \nonumber \\
&\leq \frac{2\|c\|}{\alpha} \varepsilon = K \varepsilon. \nonumber
\end{align}
Hence, we have
\begin{align}
\bigl| \min (\textbf{CP}^{\bk}) - \min (\textbf{CP}^{\bk -\varepsilon \mathbf{1}}) \bigr| \leq \frac{\kappa}{3}. \label{rate:av:neweq14}
\end{align}
Furthermore, by (\ref{min:ave:grid}), Theorem~\ref{main5}, and Lemma~\ref{upper:cost}, we also have
\begin{align}
&\bigl| \min (\textbf{CP}^{\bk-\varepsilon \mathbf{1}}) - \min (\overline{\textbf{CP}}_n^{\bk-\varepsilon \mathbf{1}}) \bigl| \leq \frac{\kappa}{3} \label{rate:av:neweq15} \\
&\bigl| \langle \mu_{\overline{\varphi}_n}, c_{\overline{\varphi}_n} \rangle - \langle \bar{\mu}^n_{\overline{\varphi}_n}, b_{n,\overline{\varphi}_n} \rangle \bigr| \nonumber \\
&\phantom{xxxxxxxxxx}= \bigl| \langle \mu_{\overline{\varphi}_n}, c_{\overline{\varphi}_n} \rangle  - \min(\textbf{CP}_n^{\bk-\varepsilon \mathbf{1}}) \bigl| \leq \frac{\kappa}{3} \label{rate:av:neweq16}  \\
&\bigl| \langle \mu_{\overline{\varphi}_n}, d_{l,\overline{\varphi}_n} \rangle - \langle \bar{\mu}^n_{\overline{\varphi}_n}, r_{l,n,\overline{\varphi}_n} \rangle \bigr| \leq \frac{\kappa}{3} \label{rate:av:neweq17}
\end{align}
where $\overline{\varphi}_n$ is the optimal policy for ($\overline{\textbf{CP}}_n^{\bk-\varepsilon \mathbf{1}}$) obtained by extending the optimal policy $\varphi_n$ of (\textbf{CP}$_n^{\bk-\varepsilon \mathbf{1}}$) to $\sX$, i.e., $\overline{\varphi}_n(\,\cdot\,|x) = \varphi(\,\cdot\,|Q_n(x))$.
Here, (\ref{rate:av:neweq15}) follows from Theorem~\ref{main5}; (\ref{rate:av:neweq16}) and (\ref{rate:av:neweq17}) follow from Lemma~\ref{upper:cost}. This completes the proof.

\end{proof}

\section{Conclusion}\label{conc}
In this paper, the approximation of discrete-time constrained Markov decision processes with compact state spaces is considered. By formulating the constrained discounted problem as linear program, we first showed that the optimal value of the reduced model asymptotically converges to the optimal value of the original model. Then, under the total variation continuity of the transition probability, we developed a method which results in approximately optimal policies. Under drift and minorization conditions on the transition probability, we derived analogous approximation results for the average cost. Under the Lipschitz continuity of the transition probability and the one-stage cost and constraint functions, explicit bounds were also derived on the performance loss due to discretization in terms of the number of grid points. In the future, we plan to extend these results to constrained Markov decision processes with unbounded state spaces.

\section*{Appendix}

\appsec

\subsection{Proof of Lemma~\ref{lemma1}}\label{prooflemma1}

We first prove that any policy $\overline{\varphi} \in \Phi$, which is an extension (to $\sX$) of a feasible policy $\varphi \in \Phi_n$ for (\textbf{CP}$_n^{\bk}$), is also feasible for ($\overline{\textbf{CP}}_n^{\bk}$); that is, it satisfies the constraints in ($\overline{\textbf{CP}}_n^{\bk}$).

Fix any $\varphi \in \Phi_n$ feasible for (\textbf{CP}$_n^{\bk}$) and extend $\varphi$ to $\sX$ by letting $\overline{\varphi}(\,\cdot\,|x) = \varphi(\,\cdot\,|Q_n(x))$. Let $\zeta \in \P(\sX_n\times\sA)$ denote
the $\beta$-discount expected occupation measure of $\varphi$, which can be disintegrated as $\zeta(dx,da) = \varphi(da|x) \hat{\zeta}(dx) \eqqcolon \varphi \otimes \hat{\zeta}$. Hence, $\zeta$ satisfies
\begin{align}
{\T}_n \zeta = {\T}_n \bigl(\varphi \otimes \hat{\zeta} \bigr) = (1-\beta) \gamma_n \text{        }\text{         and         } \text{        } \langle \zeta,\bd_n \rangle \leq \bk. \label{neweq7}
\end{align}
Let $\zeta_e = \overline{\varphi} \otimes \hat{\zeta_e} \in \P(\sX\times\sA)$ denote the $\beta$-discount expected occupation measure corresponding to $\overline{\varphi}$. Hence, $\zeta_e$ satisfies $\overline{\T}_n \zeta_e = (1-\beta) \gamma$, or more explicitly
\begin{align}
(1-\beta) \gamma(\,\cdot\,) = \hat{\zeta_e}(\,\cdot\,) - \beta \int_{\sX\times\sA} q_n(\,\cdot\,|x,a) \overline{\varphi}(da|x) \hat{\zeta_e}(dx). \label{neweq5}
\end{align}
Note that $q_n(\,\cdot\,|x,a) = q_n(\,\cdot\,|y,a)$ and $\overline{\varphi}(\,\cdot\,|x) = \overline{\varphi}(\,\cdot\,|y)$ if $x,y$ are in the same partition. Hence, if we take the pushforward of (\ref{neweq5}) with respect to $Q_n$, we obtain
\begin{align}
&(1-\beta) \gamma_n \nonumber \\
&= Q_n\ast\hat{\zeta_e} - \beta \int_{\sX\times\sA} Q_n\ast q_n(\,\cdot\,|x,a) \overline{\varphi}(da|x) \hat{\zeta_e}(dx) \nonumber \\
&= Q_n\ast\hat{\zeta_e} - \beta \sum_{i=1}^{k_n} \int_{\sA} p_n(\,\cdot\,|x_{n,i},a) \overline{\varphi}(da|x_{n,i}) Q_n\ast \hat{\zeta_e}(x_{n,i}) \nonumber \\
&= {\T}_n \bigl( \varphi \otimes Q_n\ast \hat{\zeta_e} \bigl). \nonumber
\end{align}
This and (\ref{neweq7}) imply that $Q_n\ast \hat{\zeta_e} = \hat{\zeta}$. Thus, we have
\begin{align}
\langle \zeta_e, \br_n \rangle &= \langle \zeta, \bd_n \rangle \leq \bk, \label{neweq8} \\
\intertext{and}
\langle \zeta_e, b_n \rangle &= \langle \zeta,c_n \rangle, \label{neweq9}
\end{align}
where $(\ref{neweq8})$ states that $\overline{\varphi}$ is feasible for ($\overline{\textbf{CP}}_n^{\bk}$), and (\ref{neweq9}) states that cost functions of $\varphi$ and $\overline{\varphi}$ are the same. Hence, we have
$\min (\textbf{CP}_n^{\bk}) \geq \inf (\overline{\textbf{CP}}_n^{\bk})$.

Therefore, to prove the lemma, it is enough to prove
\begin{align}
\sup (\textbf{CP}_n^{*,\bk}) = \sup (\overline{\textbf{CP}}_n^{*,\bk}), \nonumber
\end{align}
since $\sup (\textbf{CP}_n^{*,\bk}) = \min (\textbf{CP}_n^{\bk}) \geq \inf (\overline{\textbf{CP}}_n^{\bk}) \geq \sup (\overline{\textbf{CP}}_n^{*,\bk})$. Recall that we can write
\begin{align}
(\textbf{CP}_n^{*,\bk}) \text{                         }&\text{maximize}_{\bdelta \in \R^q_-} (1-\beta)\langle \gamma_n,u^*_{n,\bdelta} \rangle + \langle \bk,\bdelta \rangle, \nonumber \\
(\overline{\textbf{CP}}_n^{*,\bk}) \text{                         }&\text{maximize}_{\bdelta \in \R^q_-} (1-\beta)\langle \gamma,\overline{u}^*_{n,\bdelta} \rangle + \langle \bk,\bdelta \rangle. \nonumber
\end{align}
Since $\overline{u}^*_{n,\bdelta} = u^*_{n,\bdelta} \circ Q_n$, we have
\begin{align}
\langle \gamma, \overline{u}^*_{n,\bdelta} \rangle = \langle \gamma_n, u^*_{n,\bdelta} \rangle, \nonumber
\end{align}
for all $\bdelta \in \R^q_-$. Thus,
\begin{align}
\sup (\textbf{CP}_n^{*,\bk}) = \sup (\overline{\textbf{CP}}_n^{*,\bk}). \nonumber
\end{align}

\subsection{Proof of Lemma~\ref{lemma2}}\label{prooflemma2}

Let $\P$ and $\P_n$ denote the set of $\beta$-discount expected occupation measures for MDP and $\overline{\text{MDP}}_n$, respectively. It can be proved that for each $l=1,\ldots,q$, we have
\begin{align}
\inf_{\zeta \in \P} \langle \zeta,d_l \rangle &= (1-\beta) \langle J^*_l, \gamma \rangle, \nonumber \\
\inf_{\zeta \in \P_n} \langle \zeta,r_{l,n} \rangle &= (1-\beta) \langle J^*_{l,n}, \gamma \rangle, \nonumber \\
\intertext{where}
J^*_l(x) &= \inf_{\pi \in \Pi} (1-\beta) \cE_{x}^{\pi} \biggl[ \sum_{t=0}^{\infty} \beta^t d_l(X_t,A_t) \biggr], \nonumber \\
J^*_{l,n}(x) &= \inf_{\pi \in \Pi} (1-\beta) \cE_{x}^{\pi} \biggl[ \sum_{t=0}^{\infty} \beta^t r_{l,n}(X_t,A_t) \biggr]. \nonumber
\end{align}
By \cite[Theorem 2.4]{SaYuLi17}, we have $\|J^*_{l,n} - J^*_l\| \rightarrow 0$ as $n\rightarrow\infty$, and therefore, $|\inf_{\zeta \in \P_n} \langle \zeta,r_{l,n} \rangle - \inf_{\zeta \in \P} \langle \zeta,d_l \rangle| \rightarrow 0$ as $n\rightarrow\infty$. By Assumption~\ref{as0}-(c), we have
\begin{align}
\inf_{\zeta \in \P} \langle \zeta,d_l \rangle + \alpha'_l \leq k_l \text{    } \text{for  } l=1,\ldots,q. \nonumber
\end{align}
Thus, one can find $n(\bk)\geq1$ large enough such that
\begin{align}
\inf_{\zeta \in \P_n} \langle \zeta,r_{l,n} \rangle + \frac{\alpha'_l}{2} \leq k_l \text{    } \text{for  } l=1,\ldots,q. \nonumber
\end{align}

\subsection{Proof of Lemma~\ref{lemma3}}\label{prooflemma3}

By induction we first prove that for any $t\geq1$,
\begin{align}
\lim_{n\rightarrow\infty}\bigl\| R_n^t(\,\cdot\,|\gamma) - P_n^t(\,\cdot\,|\gamma) \bigr\|_{TV} = 0. \label{neweq13}
\end{align}
For $t=1$ the claim holds by the following argument:
\begin{align}
&\bigl\| R_n(\,\cdot\,|\gamma) - P_n(\,\cdot\,|\gamma) \bigr\|_{TV} \nonumber \\
&\leq \int_{\sX} \int \bigl\| P_n(\,\cdot\,|z) - P_n(\,\cdot\,|x) \bigr\|_{TV} \nu_{n,i_n(x)}(dz) \gamma(dx)\nonumber \\
&\leq \int_{\sX\times\sA} \int \bigl\| p(\,\cdot\,|z,a) - p(\,\cdot\,|x,a) \bigr\|_{TV} \nonumber \\
&\phantom{xxxxxxxxxxxxxxxxxxxxxxx}\nu_{n,i_n(x)}(dz) \overline{\varphi}_n(da|x) \gamma(dx) \nonumber \\
&\phantom{xxxxxxxxxxxx} \bigl(\text{since } \text{  }  \overline{\varphi}_n(da|z) = \overline{\varphi}_n(da|x), \text{  } z\in\S_{n,i_n(x)} \bigr) \nonumber \\
&\leq \sup_{(x,a) \in \sX\times\sA} \sup_{z \in \S_{n,i_n(x)}} \bigl\|p(\,\cdot\,|z,a) - p(\,\cdot\,|x,a) \bigr\|_{TV}. \nonumber
\end{align}
As the mapping $p(\,\cdot\,|x,a):\sX\times\sA \rightarrow \P(\sX)$ is (uniformly) continuous, the result follows. Assume the claim is true for $t\geq1$. Then we have
\begin{align}
&\bigl\| R_n^{t+1}(\,\cdot\,|\gamma) - P_n^{t+1}(\,\cdot\,|\gamma) \bigr\|_{TV} \nonumber \\
&= \biggl\| \int_{\sX} R_n(\,\cdot\,|x) R_n^t(dx|\gamma) - \int_{\sX} P_n(\,\cdot\,|x) P_n^t(dx|\gamma) \biggr\|_{TV} \nonumber \\
&\leq \biggl\| \int_{\sX} R_n(\,\cdot\,|x) R_n^t(dx|\gamma) - \int_{\sX} P_n(\,\cdot\,|x) R_n^t(dx|\gamma) \biggr\|_{TV} \nonumber \\
&\phantom{xxxx}+ \biggl\| \int_{\sX} P_n(\,\cdot\,|x) R_n^t(dx|\gamma) - \int_{\sX} P_n(\,\cdot\,|x) P_n^t(dx|\gamma) \biggr\|_{TV} \nonumber \\
&\leq \int_{\sX} \bigl\| R_n(\,\cdot\,|x) - P_n(\,\cdot\,|x) \bigr\|_{TV} R_n^t(dx|\gamma) \nonumber \\
&\phantom{xxxxxxxxxxxxxxxxxxxxx}+  \bigl\| R_n^t(\,\cdot\,|\gamma)-P_n^t(\,\cdot\,|\gamma) \bigr\|_{TV} \nonumber \\
&\leq \sup_{(x,a) \in \sX\times\sA} \sup_{z \in \S_{n,i_n(x)}} \bigl\| p(\,\cdot\,|z,a) - p(\,\cdot\,|x,a) \bigr\|_{TV} \nonumber \\
&\phantom{xxxxxxxxxxxxxxxxxxxxx}+  \bigl\| R_n^t(\,\cdot\,|\gamma)-P_n^t(\,\cdot\,|\gamma) \bigr\|_{TV}. \nonumber
\end{align}
Since the mapping $p(\,\cdot\,|x,a):\sX\times\sA \rightarrow \P(\sX)$ is uniformly continuous, the first term converges to zero. The second term also converges to zero since the claim holds for $t$. This completes the proof of (\ref{neweq13}).

Using (\ref{neweq13}), we obtain
\begin{align}
&\biggl| \int_{\sX} g_{n,\overline{\varphi}_n}(x) R_n^t(dx|\gamma) - \int_{\sX} g_{\overline{\varphi}_n}(x) P_n^t(dx|\gamma) \biggr|  \nonumber \\
&\leq \biggl| \int_{\sX} g_{n,\overline{\varphi}_n}(x) R_n^t(dx|\gamma) - \int_{\sX} g_{\overline{\varphi}_n}(x) R_n^t(dx|\gamma) \biggr| \nonumber \\
&\phantom{xxxxxxx} + \biggl| \int_{\sX} g_{\overline{\varphi}_n}(x) R_n^t(dx|\gamma) - \int_{\sX} g_{\overline{\varphi}_n}(x) P_n^t(dx|\gamma) \biggr| \nonumber \\
&\leq \|g_n-g\| + \|g\| \text{ } \bigl\| R_n^t(\,\cdot|\gamma)-P_n^t(\,\cdot|\gamma) \bigr\|_{TV}. \nonumber
\end{align}
As $g_n \rightarrow g$ uniformly and $\bigl\| R_n^t(\,\cdot\,|\gamma)-P_n^t(\,\cdot\,|\gamma) \bigr\|_{TV} \rightarrow 0$, the result follows.

\subsection{Proof of Lemma~\ref{lemma4}}\label{prooflemma4}

Analogous to the proof of Lemma~\ref{lemma1}, we first prove that any policy $\overline{\varphi} \in \Phi$, which is an extension (to $\sX$) of a feasible policy $\varphi \in \Phi_n$ for $($\textbf{CP}$_n^{\bk})$, is also feasible for $(\overline{\textbf{CP}}_n^{\bk})$.

Fix any $\varphi \in \Phi_n$ feasible for $($\textbf{CP}$_n^{\bk})$ and extend $\varphi$ to $\sX$ by letting $\overline{\varphi}(\,\cdot\,|x) = \varphi(\,\cdot\,|Q_n(x))$. It can be proved that
\begin{align}
\mu^n_{\varphi}(\,\cdot\,) = Q_n \ast \bar{\mu}^n_{\overline{\varphi}}. \nonumber
\end{align}
Then we have
\begin{align}
 \langle \bar{\mu}^n_{\overline{\varphi}}, b_{n,\overline{\varphi}} \rangle &= \langle \mu^n_{\varphi}, c_{n,\varphi} \rangle \nonumber \\
 \langle \bar{\mu}^n_{\overline{\varphi}}, r_{l,n,\overline{\varphi}} \rangle &= \langle \mu^n_{\varphi}, d_{l,n,\varphi} \rangle, \text{ } l=1,\ldots,q, \nonumber
\end{align}
which proves the result. Therefore, we have $\min(\textbf{CP}_n^{\bk}) \geq \min(\overline{\textbf{CP}}_n^{\bk})$.

To complete the proof it is sufficient to prove $\min(\textbf{CP}_n^{\bk}) \leq \min(\overline{\textbf{CP}}_n^{\bk})$. This can be done by formulating both $(\textbf{CP}_n^{\bk})$ and $(\overline{\textbf{CP}}_n^{\bk})$ as linear programs \cite{HeGoLo03} and then proving that the dual linear programs satisfy $\sup(\textbf{CP}^{*,\bk}_n) \leq \sup(\overline{\textbf{CP}}^{*,\bk}_n)$. As there is no duality gap for both problems \cite[Theorem 4.4]{HeGoLo03}, we have $\min(\textbf{CP}_n^{\bk}) = \sup(\textbf{CP}^{*,\bk}_n) \leq \sup(\overline{\textbf{CP}}^{*,\bk}_n) = \min(\overline{\textbf{CP}}_n^{\bk})$. Indeed, the dual problem
$(\overline{\textbf{CP}}^{*,\bk}_n)$ is given by
\begin{align}
\text{maximize}_{(\bdelta,\delta_0,u) \in \R_{-}^q \times \R \times B(\sX)} \text{ }
\delta_0 + \bigl\langle \bk, \bdelta \bigr\rangle \nonumber
\end{align}
subject to
\begin{align}
u(x) + \delta_0 \leq b_n(x,a) - \sum_{l=1}^q \delta_l r_{l,n}(x,a) + \int_{\sX} u(y) q_n(dy|x,a) \nonumber
\end{align}
for all $(x,a) \in \sX\times\sA$.
And, the dual problem $(\textbf{CP}^{*,\bk}_n)$ is given by
\begin{align}
\text{maximize}_{(\bdelta,\delta_0,u) \in \R_{-}^q \times \R \times B(\sX_n)} \text{ }
\delta_0 + \bigl\langle \bk, \bdelta \bigr\rangle \nonumber
\end{align}
subject to
\begin{align}
u(x) + \delta_0 \leq c_n(x,a) - \sum_{l=1}^q \delta_l d_{l,n}(x,a) + \int_{\sX_n} u(y) p_n(dy|x,a) \nonumber
\end{align}
for all $(x,a) \in \sX_n\times\sA$. It is straightforward to prove that if the triple $(\bdelta,\delta_0,u)$ is feasible for $(\textbf{CP}^{*,\bk}_n)$, then the triple $(\bdelta,\delta_0,u\circ Q_n)$ is also feasible for $(\overline{\textbf{CP}}^{*,\bk}_n)$. This implies that
\begin{align}
\sup(\textbf{CP}^{*,\bk}_n) \leq \sup(\overline{\textbf{CP}}^{*,\bk}_n). \nonumber
\end{align}


\subsection{Proof of Proposition~\ref{upper:prop2}} \label{proofupper:prop2}

By induction we first prove that for any $t\geq1$,
\begin{align}
\bigl\| R_n^t(\,\cdot\,|\gamma) - P_n^t(\,\cdot\,|\gamma) \bigr\|_{TV} \leq t G_p 2 \alpha (1/n)^{1/d} . \label{upper:neweq13}
\end{align}
For $t=1$ the claim holds by the following argument:
\begin{align}
&\bigl\| R_n(\,\cdot\,|\gamma) - P_n(\,\cdot\,|\gamma) \bigr\|_{TV} \nonumber \\
&\leq \sup_{(x,a) \in \sX\times\sA} \sup_{z \in \S_{n,i_n(x)}} \bigl\|p(\,\cdot\,|z,a) - p(\,\cdot\,|x,a) \bigr\|_{TV} \nonumber \\
&\leq G_p \sup_{(x,a) \in \sX\times\sA} \sup_{z \in \S_{n,i_n(x)}} d_{\sX}(x,z) \nonumber \\
&\leq G_p 2\alpha (1/n)^{1/d}. \nonumber
\end{align}
Assume the claim is true for $t\geq1$. Then we have
\begin{align}
&\bigl\| R_n^{t+1}(\,\cdot\,|\gamma) - P_n^{t+1}(\,\cdot\,|\gamma) \bigr\|_{TV} \nonumber \\
&\leq \sup_{(x,a) \in \sX\times\sA} \sup_{z \in \S_{n,i_n(x)}} \bigl\| p(\,\cdot\,|z,a) - p(\,\cdot\,|x,a) \bigr\|_{TV} \nonumber \\
&\phantom{xxxxxxxxxxxxxxxxxxxxx}+  \bigl\| R_n^t(\,\cdot\,|\gamma)-P_n^t(\,\cdot\,|\gamma) \bigr\|_{TV}  \nonumber \\
&\leq G_p 2\alpha (1/n)^{1/d} + t G_p 2\alpha (1/n)^{1/d}. \nonumber
\end{align}
This completes the proof of (\ref{upper:neweq13}).

Using (\ref{upper:neweq13}), for any $t$, we obtain
\begin{align}
&\biggl| \int_{\sX} g_{n,\overline{\varphi}_n}(x) R_n^t(dx|\gamma) - \int_{\sX} g_{\overline{\varphi}_n}(x) P_n^t(dx|\gamma) \biggr|  \nonumber \\
&\leq \biggl| \int_{\sX} g_{n,\overline{\varphi}_n}(x) R_n^t(dx|\gamma) - \int_{\sX} g_{\overline{\varphi}_n}(x) R_n^t(dx|\gamma) \biggr| \nonumber \\
&\phantom{xxxxxxx} + \biggl| \int_{\sX} g_{\overline{\varphi}_n}(x) R_n^t(dx|\gamma) - \int_{\sX} g_{\overline{\varphi}_n}(x) P_n^t(dx|\gamma) \biggr| \nonumber \\
&\leq \|g_n-g\| + \|g\| \text{ } \bigl\| R_n^t(\,\cdot\|\gamma)-P_n^t(\,\cdot\|\gamma) \bigr\|_{TV}  \nonumber \\
&\leq K_g \sup_{(x,a) \in \sX\times\sA} \sup_{z \in \S_{n,i_n(x)}} d_{\sX} (x,z) + \|g\| t G_p 2\alpha (1/n)^{1/d} \nonumber \\
&\leq K_g 2\alpha (1/n)^{1/d} + \|g\| t G_p 2\alpha (1/n)^{1/d}. \nonumber
\end{align}
Then, we have
\begin{align}
&\bigl| J^{g_n}_n(\overline{\varphi}_n,\gamma) - J^g(\overline{\varphi}_n,\gamma) \bigr| \nonumber \\
&\leq (1-\beta) \sum_{t=0}^{\infty} \beta^t \biggl| \int_{\sX} g_{n,\overline{\varphi}_n}(x) R_n^t(dx|\gamma) \nonumber \\
&\phantom{xxxxxxxxxxxxxxxxx}- \int_{\sX} g_{\overline{\varphi}_n}(x) P_n^t(dx|\gamma) \biggr|\nonumber \\
&\leq (1-\beta) \sum_{t=0}^{\infty} \beta^t \bigl( K_g 2\alpha (1/n)^{1/d} + \|g\| t G_p 2\alpha (1/n)^{1/d} \bigr) \nonumber \\
&= (1-\beta) \biggl( \frac{K_g 2\alpha}{1-\beta} (1/n)^{1/d} + \frac{\|g\| G_p 2\alpha}{(1-\beta)^2} (1/n)^{1/d} \biggr)\nonumber \\
&= \biggl( K_g + \frac{\|g\| G_p}{1-\beta} \biggr) 2 \alpha (1/n)^{1/d}. \nonumber
\end{align}

\subsection{Proof of Lemma~\ref{upper:cost}}\label{proofupper:cost}

We only prove the result for $c$ since proof for any $d_l$ is identical. Note that similar to the proof of (\ref{upper:neweq13}), one can show that, for any $t$,
\begin{align}
\sup_{(y,\varphi) \in \sX\times\Phi} \bigl\|p^t(\,\cdot\,|y,\varphi) - q_n^t(\,\cdot\,|y,\varphi) \bigr\|_{TV}
\leq t G_p 2 \alpha (1/n)^{1/d}. \nonumber
\end{align}
Furthermore, we have
\begin{align}
&\sup_{\varphi \in \Phi} \bigl| \langle \mu_{\varphi}, c_{\varphi} \rangle - \langle \bar{\mu}^n_{\varphi}, b_{n,\varphi} \rangle \bigr| \nonumber \\
&\leq \sup_{\varphi \in \Phi} \bigl| \langle \mu_{\varphi}, c_{\varphi} \rangle - \langle \bar{\mu}^n_{\varphi}, c_{\varphi} \rangle \bigr| + \sup_{\varphi \in \Phi} \bigl| \langle \bar{\mu}^n_{\varphi}, c_{\varphi} \rangle - \langle \bar{\mu}^n_{\varphi}, b_{n,\varphi} \rangle \bigr|. \nonumber
\end{align}
It is straightforward to show that the second term in the right side of the above equation can be upper bounded as
\begin{align}
\sup_{\varphi \in \Phi} \bigl| \langle \bar{\mu}^n_{\varphi}, c_{\varphi} \rangle - \langle \bar{\mu}^n_{\varphi}, b_{n,\varphi} \rangle \bigr| \leq K_c 2 \alpha (1/n)^{1/d}. \nonumber
\end{align}
For the first term, we have, for any $t\geq1$ and $y \in \sX$,
\begin{align}
&\sup_{\varphi \in \Phi} \bigl| \langle \mu_{\varphi}, c_{\varphi} \rangle - \langle \bar{\mu}^n_{\varphi}, c_{\varphi} \rangle \bigr| \nonumber \\
&= \sup_{\varphi \in \Phi} \biggl | \int_{\sX} c_{\varphi}(x) \mu_{\varphi}(dx) - \int_{\sX} c_{\varphi}(x) \bar{\mu}^n_{\varphi}(dx) \biggr | \nonumber \\
&\leq \sup_{\varphi \in \Phi} \biggl | \int_{\sX} c_{\varphi}(x)  \mu_{\varphi}(dx) - \int_{\sX} c_{\varphi}(x) p^t(dx|y,\varphi) \biggr | \nonumber \\
&\phantom{xxxx}+ \sup_{\varphi \in \Phi} \biggl | \int_{\sX} c_{\varphi}(x) p^t(dx|y,\varphi) - \int_{\sX} c_{\varphi}(x) q_n^t(dx|y,\varphi) \biggr | \nonumber \\
&\phantom{xxxxxxxxx}+ \sup_{\varphi \in \Phi} \biggl | \int_{\sX} c_{\varphi}(x) q_n^t(dx|y,\varphi) - \int_{\sX} c_{\varphi}(x) \bar{\mu}^n_{\varphi}(dx) \biggr | \nonumber \\
&\leq 2 R \kappa^t \|c\|  + \|c\| \sup_{(y,\varphi)\in\sX\times\Phi} \bigl\| q_n^t(\,\cdot\,|y,\varphi) - p^t(\,\cdot\,|y,\varphi) \bigr\|_{TV} \nonumber \\
&\leq 2 R \kappa^t \|c\| + \|c\| t G_p 2 \alpha (1/n)^{1/d}, \nonumber
\end{align}
where $R$ and $\kappa$ are the constants in Theorem~\ref{compact:thm4}. Therefore, we obtain
\begin{align}
&\sup_{\varphi \in \Phi} \bigl| \langle \mu_{\varphi}, c_{\varphi} \rangle - \langle \bar{\mu}^n_{\varphi}, b_{n,\varphi} \rangle \bigr| \nonumber \\
&\phantom{xxxxxxxxxx}\leq I_1 \kappa^t + I_2 (1/n)^{1/d} + I_3 (1/n)^{1/d} t. \label{rateofconvergence}
\end{align}
To obtain the upper bound that only depends on $n$, the dependence of the upper bound on $t$ has to be written as a function of $n$. This can be done by (approximately) minimizing the upper bound in (\ref{rateofconvergence}) with respect to $t$ for each $n$. For each $n$, it is straightforward to compute that
\begin{align}
t'(n) \coloneqq \ln\bigl(\frac{n^{1/d}}{I_4}\bigr) \frac{1}{ \ln(\frac{1}{\kappa})} \nonumber
\end{align}
is the zero of the derivative of the convex function of $n$ given in (\ref{rateofconvergence}). Letting $t = \lceil t'(n) \rceil$ in (\ref{rateofconvergence}), we obtain the upper bound in the lemma.

\subsection{Proof of Lemma~\ref{lemma6}}\label{prooflemma6}

The proof Lemma~\ref{lemma6} follows the arguments in the proof of \cite[Theorems 3.6 and 4.10]{DuPr13}. Let
\begin{align}
\C = \bigcup_{\varphi \in \Phi} \biggl \{ \bh \in \R^q: \langle \mu_{\varphi}, \bd \rangle + \balpha = \bh \text{  for some  } \balpha \in \R^q_+ \biggr\} \nonumber
\end{align}
and define the function $\V: \C \rightarrow \R$ by
\begin{align}
&\V(\bh) \nonumber \\
&= \min \biggl\{ \langle \mu_{\varphi},c \rangle: \varphi \in \Phi \text{  and  } \langle \mu_{\varphi}, \bd \rangle + \balpha = \bh, \text{   } \balpha \in \R^q_+ \biggr\}. \nonumber
\end{align}
Note that $\V(\bm) = \min(\textbf{CP}^{\bm})$. It can be proved that $\C$ is a convex subset of $\R^q$ and $\V$ is a convex function. Since $\bm \in \intr \C$, by \cite[Theorem 7.12]{AlBo06} there exists $\bdelta^* \in \R^q$ such that for all $\bh \in \C$
\begin{align}
\V(\bh) - \V(\bm) \geq - \langle \bdelta^*, \bh - \bm \rangle, \nonumber
\end{align}
i.e., $\bdelta^*$ is a subgradient of $\V$ at $\bm$. Since $\V(\bh) \leq \V(\bm)$ when $\bh \geq \bm$, we have $\bdelta^* \geq \mathbf{0}$.

For any $\varphi \in \Phi$, we have $\langle \mu_{\varphi},\bd \rangle \in \C$ with $\balpha = \mathbf{0}$ and $\V(\langle \mu_{\varphi},\bd \rangle) \leq \langle \mu_{\varphi},c \rangle$. Hence,
\begin{align}
\langle \mu_{\varphi},c \rangle - \V(\bm) \geq -\bigl\langle \bdelta^*, \langle \mu_{\varphi},\bd \rangle - \bm \bigr\rangle \nonumber
\end{align}
for all $\mu_{\varphi} \in \Phi$ and therefore,
\begin{align}
\min (\textbf{CP}^{\bm}) &= \V(\bm) \leq \inf_{\varphi \in \Phi} \bigl\langle \mu_{\varphi}, c+\langle \bdelta^*,\bd \rangle \bigr\rangle - \langle \bdelta^*, \bm \rangle \nonumber \\
&\eqqcolon \delta_0^{*} - \langle \bdelta^*, \bm \rangle. \label{dual}
\end{align}
Note that by Theorem~\ref{compact:thm4}, there exists $h^* \in B(\sX)$ such that the following average cost optimality equality (ACOE) holds:
\begin{align}
&\delta_0^{*} + h^{*}(x) \nonumber \\
&= \min_{a\in\sA} \biggl[ c(x,a) + \sum_{l=1}^q \delta^*_l d_l(x,a) + \int_{\sX} h^{*}(y) p(dy|x,a) \biggr]. \nonumber
\end{align}
Hence, the triple $(\delta_0^*,-\bdelta^*,h^*)$ is a feasible point for the dual problem
$(\textbf{CP}^{*,\bm})$ which is given by
\begin{align}
\text{maximize }
\delta_0 + \bigl\langle \bm, \bdelta \bigr\rangle \nonumber
\end{align}
subject to
\begin{align}
u(x) + \delta_0 \leq c(x,a) - \sum_{l=1}^q \delta_l d_{l}(x,a) + \int_{\sX} u(y) p(dy|x,a) \nonumber
\end{align}
for all $(x,a) \in \sX\times\sA$, $\bdelta \in \R_{-}^q$, $\delta_0 \in \R$, and $u \in B(\sX)$. Therefore, we can bound (\ref{dual}) as follows:
\begin{align}
(\ref{dual}) \leq \sup (\textbf{CP}^{*,\bm}) = \min (\textbf{CP}^{\bm}). \nonumber
\end{align}
Thus, $(\delta_0^*, -\bdelta^*)$ is the maximizer for ($\textbf{CP}^{*,\bm}$). Note that $\bm - \frac{\alpha}{2} \mathbf{1} \in \C$ and therefore,
\begin{align}
\V(\bm - \frac{\alpha}{2} \mathbf{1}) - \V(\bm) \geq -\langle \bdelta^*, -\frac{\alpha}{2} \mathbf{1} \rangle &= \frac{\alpha}{2} \|\bdelta^*\|_1. \nonumber
\end{align}
Since $\V(\bm - \frac{\alpha}{2} \mathbf{1}) \leq \|c\|$ and $\V(\bm) \geq 0$, we have
\begin{align}
\|\bdelta^*\|_1 \leq \frac{ 2 \|c\|}{\alpha}. \nonumber
\end{align}




\end{document}